\documentclass[a4paper]{article}

\usepackage{lipsum}
\usepackage{graphicx}
\usepackage{epstopdf}
\usepackage{algorithmic}
\usepackage{hhline}
\usepackage{multirow}
\usepackage {tikz}
\usepackage{bm}
\usetikzlibrary{calc, positioning, decorations.pathreplacing}
\usepackage{etoolbox}
\usepackage{todonotes}

\usepackage{mathtools}
\usepackage{appendix}
\usepackage{amsmath,amssymb,amsfonts,dsfont,mathrsfs,environ}
\usepackage{amsopn}
\usepackage{algorithm}
\usepackage{algorithmic}
\usepackage{color}
\usepackage{authblk}
\usepackage{titlesec}
\usepackage[a4paper, total={6.2in, 8.5in}]{geometry}
\usepackage{fancyhdr,lipsum}
\usepackage{hyperref}
\usepackage[square,sort,comma,numbers]{natbib}
\usepackage{times}

\ifpdf
\DeclareGraphicsExtensions{.eps,.pdf,.png,.jpg}
\else
\DeclareGraphicsExtensions{.eps}
\fi

\title{Newton correction methods for computing real eigenpairs of symmetric tensors}
\author[1]{Ariel Jaffe}
\author[1]{Roi Weiss}
\author[1]{Boaz Nadler}
\affil[1]{Weizmann Institute of Science, Israel}
\affil[ ]{\textit {\{ariel.jaffe, roi.weiss, boaz.nadler\}@weizmann.ac.il}}
\date{}
\DeclareMathOperator{\diag}{diag}

\DeclareMathOperator*{\argmin}{argmin}
\DeclarePairedDelimiter\floor{\lfloor}{\rfloor}

\newcommand{\T}{\mathcal T}

\newcommand{\R}{\mathbb R}
\newcommand{\A}{\mathbb A}
\newcommand{\Av}{\bm 1_{\mathbb A}}
\newcommand{\Avc}{\bm 1_{\mathbb A^c}}
\newcommand{\nrm}[1]{{\|#1\|}}
\newcommand{\eps}{\varepsilon}

\newcommand{\xv}{\bm x}
\newcommand{\xk}{\bm x_{(k)}}
\newcommand{\xo}{\bm x_{(1)}}
\newcommand{\xz}{\bm x_{(0)}}
\newcommand{\xkpo}{\bm x_{(k+1)}}
\newcommand{\xast}{\bm x^\ast}
\newcommand{\yv}{\bm y}
\newcommand{\yz}{\bm y_{(0)}}
\newcommand{\yast}{\bm y^*}

\newcommand{\yk}{\bm y_{(k)}}

\newcommand{\err}{e}

\newcommand{\lamast}{\lambda^\ast}
\newcommand{\g}{\gamma}
\newcommand{\bz}{\bm z}
\newcommand{\tA}{H}
\newcommand{\grad}{\bm g}

\newcommand{\ortc}{\bm u}
\newcommand{\ortcast}{\ortc^*}
\newcommand{\im}{i_1,\ldots,i_m}
\newcommand{\xim}{x_{i_1}\ldots x_{i_m}}
\newcommand{\xlocal}{\xv_{\text{loc}}}

\newcommand{\beq}{\begin{eqnarray*}}
	\newcommand{\eeq}{\end{eqnarray*}}
\newcommand{\beqn}{\begin{eqnarray}}
\newcommand{\eeqn}{\end{eqnarray}}
\NewEnviron{salign}{%
	\begin{align}\begin{split}
			\BODY
	\end{split}\end{align}
}

\newtheorem{theorem}{Theorem}
\newtheorem{proof}{Proof}
\newtheorem{lemma}{Lemma}
\newtheorem{proposition}{Proposition}
\newtheorem{conjecture*}{Conjecture}
\newtheorem{definition}{Definition}
\newtheorem{remark}{Remark}

\titleformat*{\section}{\Large\bfseries}
\providecommand{\keywords}[1]{\textbf{\textit{Key words.}} #1}

\begin{document}

\maketitle
%
\begin{abstract}
Real eigenpairs of symmetric tensors play an important role in multiple applications. 
In this paper we propose and analyze a fast iterative Newton-based method to compute 
real eigenpairs of symmetric tensors. 
We derive sufficient conditions for  
a real eigenpair to be a stable fixed point for our method,
and prove that given a sufficiently close initial guess, the 
convergence rate is quadratic.
Empirically, our method converges to a significantly larger number of eigenpairs compared to previously proposed iterative methods, and with enough random initializations typically finds all real eigenpairs.
In particular,
for a {generic} symmetric tensor, the sufficient conditions for local convergence of our Newton-based method hold simultaneously for all its real eigenpairs.

\end{abstract}
%

%
%
\section{Introduction}
%

Eigenpairs of symmetric tensors have received much attention in recent years due to their applicability in a wide range of disciplines.
Introduced by Lim \cite{limsingular} and Qi \cite{qi2005eigenvalues}, 
tensor eigenpairs were used for example in the analysis of hypergraphs \cite{li2013z}, high-order Markov chains \cite{ng2009finding}, establishing the positive-definiteness of multivariate forms \cite{ni2008eigenvalue}, diffusion tensor imaging \cite{schultz2014higher,qi2008d}, and data analysis \cite{anandkumar2012method,anandkumar2014tensor}.

The focus of this paper is on fast iterative methods to compute the real eigenpairs of symmetric tensors.
These methods were recently applied by the authors and collaborators in  \cite{Jaffe2018learning} for learning a binary latent variable model by computing the eigenpairs of a third order moment tensor of the observed data.

There are major differences between tensor eigenpairs, whose formal definition is reviewed in Section \ref{sec:preliminaries}, and their well studied matrix counterparts.
Whereas any symmetric $n \times n$ matrix has exactly $n$ real eigenvalues with corresponding orthogonal eigenvectors,
the situation for tensors is fundamentally different. 
The eigenvectors of a symmetric tensor are not necessarily orthogonal and some may in fact be complex-valued.
Furthermore, the number of eigenvalues is in general significantly larger than $n$. 
For symmetric real tensors of order $m$ and dimensionality $n$, 
\cite{cartwright2013number} proved that the number of complex eigenvalues is at most $((m-1)^n-1)/(m-2)$.
For generic tensors, this is the \textit{exact} number of complex eigenvalues.
As a lower bound, it is known that for odd-order
tensors at least one real eigenvalue exists 
\cite{cartwright2013number},
while for symmetric even-ordered tensors at least $n$ real eigenvalues exist 
\cite{chang2009eigenvalue}. 
Recently, 
 \cite{breiding2017average} analyzed the expected number of real eigenvalues of a random Gaussian symmetric tensor.

From a computational perspective, while all matrix eigenpairs can be computed efficiently, 
Hillar and Lim 
\cite{hillar2013most} showed that enumerating all eigenpairs of a symmetric tensor is in general \#P.
Nonetheless, Cui et al.\ \cite{cui2014all}  derived a method to compute \textit{all} eigenpairs sequentially, based on a hierarchy of semidefinite relaxations.
Chen et al.\ \cite{chen2016computing} proposed a homotopy continuation method for the same purpose. 
While these algorithms are guaranteed to find all isolated
 eigenpairs, they are computationally demanding even for moderate tensor dimensions. For example, computing all real eigenpairs of a random $8\times 8 \times 8 \times 8$ tensor using the
\texttt{zeig} procedure in the \texttt{TenEig} package of \cite{chen2016computing}
takes several hours on a standard PC. 

In recent years, several iterative methods were developed to compute at least some
tensor eigenpairs. Some methods were specifically designed to compute
the largest eigenvalue \cite{ng2009finding,liu2010always,l2015sequential,gautier2016tensor}.
Han \cite{han2012unconstrained} proposed a method based on unconstrained optimization to compute both the maximal and minimal eigenvalues of an even-order tensor. 
As described in Section \ref{sec:high_order}, adaptations of the popular power method to tensors were suggested in  \cite{Lathauwer_2000,kolda2011shifted,kolda2014adaptive,zhang2012best}. 
While these iterative methods are computationally fast, in general they converge to only a strict subset of all eigenpairs. 

In this work we present a different iterative approach to compute real eigenpairs
of symmetric tensors. As detailed in Section \ref{sec:newton_based}, our approach 
is based on adapting the
matrix Newton correction method (NCM) to the tensor case. We derive sufficient conditions for local convergence of NCM and prove that its convergence rate is quadratic.
Our analysis reveals that NCM may fail to converge to eigenvectors with eigenvalue zero and has small attraction region for eigenvalues close to zero. 
To overcome this limitation, we next derive a variant, denoted the orthogonal Newton correction method (O--NCM), which enjoys improved run-time and convergence guarantees.
We observe that for a generic symmetric tensor, the sufficient conditions for 
either NCM or O--NCM to converge to all its eigenpairs hold with probability one. 
In Section \ref{sec:fail} we illustrate that these sufficient conditions are 
not necessary. 

In Section \ref{sec:results} we present numerical simulations that support our theoretical analysis. 
For random tensors of modest size, multiple random initializations of NCM or O--NCM can find all eigenpairs significantly faster than other methods. For example,
on a random $8\times 8 \times 8 \times 8$  tensor, our methods typically found
all eigenpairs within a few seconds. 
We conclude with a summary and discussion in Section \ref{sec:discussion}.
\paragraph{Notation}
We denote vectors by lowercase boldface letters, as in $\xv$, matrices by uppercase letters, as in $W$, and
higher-order tensors by caligraphic letters, as in $\T$.
We denote $[n]=\{1,\dots,n\}$. $I$ is the identity matrix whose dimension depends on the context and $S_{n-1} = \{\xv\in\R^n:\nrm{\xv}=1\}$ is the unit sphere.

\section{The symmetric tensor eigen-problem}\label{sec:preliminaries}
Let $\T\in \R^{n\times\ldots\times n}$ be a tensor of order $m$ and dimension $n$, with entries $t_{\im}$, where $\im \in [n]$.
We assume that $\T$ is symmetric, namely, $t_{\im}= t_{\pi(\im)}$ for all permutations $\pi$ of the $m$ indices $\im$. 
A tensor $\T$ can be viewed as a {\em multi-linear} operator:
for matrices $W^{1},\dots,W^{m}$ with $W^{i}\in\R^{n \times d_i}$, the tensor-mode product, denoted by $\T(W^{1},\ldots,W^{m})\in\R^{d_1\times\dots\times d_m}$, yields a new tensor whose $(i_1,\dots,i_m)^{\text{th}}$ entry is
\beq
\left[\T(W^{1},\ldots,W^{m})\right]_{i_1,\dots,i_m} = 
\sum_{j_1,\dots,j_m\in[n]} 
W^{1}_{j_1,i_1} \cdots W^{m}_{j_m,i_m} t_{j_1,\dots,j_m}.
\eeq

\paragraph{Tensor eigenpairs} 
Several definitions of tensor eigenpairs appear in the literature.
Here we consider the one introduced as $Z$-eigenpairs in \cite{qi2005eigenvalues} and $l^2$-eigenpairs in \cite{limsingular}.

\begin{definition}
\label{def:eig_def}
A pair $(\xast,\lamast)\in\R^{n}\times \R$ is a real eigenpair of $\T$ if
\begin{equation}\label{eq:fixed_point_def}
\mathcal T(I,\xast,\ldots,\xast) = \lamast \xast \quad \text{and} \quad {\nrm{\xast}}=1.
\end{equation}
\end{definition}

Note that if $(\xast,\lamast)$ is an eigenpair, then  $(-\xast,(-1)^m\lamast)$ is an eigenpair as well. 
Following common practice, we treat these two pairs as belonging to the same equivalence class \cite{cartwright2013number}. 

Definition \ref{def:eig_def} can be equivalently stated using
the following $m$-degree homogeneous polynomial in $\xv \in \R^n$,
\begin{equation}\label{eq:induced_function}
\mu(\xv) = \mathcal T(\xv,\ldots,\xv) = \sum_{\im\in[n]} t_{\im}
x_{i_1} \cdots  x_{i_m}.
\end{equation}
As shown in \cite{limsingular}, the real eigenpairs of $\T$ correspond to the {\em critical points} of $\mu(\xv)$ when constrained to the unit sphere  
$S_{n-1}$.
Formally, 
define the Lagrangian
\begin{equation}\label{eq:lagrangian}
L(\xv,\lambda) = \mu(\xv)-
\frac{m\lambda}{2}({\nrm{\xv}}^{2}-1),
\qquad \lambda \in\R.
\end{equation}
A constrained critical point $(\xast,\lamast)$ of $\mu$ satisfies the Karush-Kuhn-Tucker conditions,
\beq
\frac{1}{m}\nabla_{\xast}  L(\xast,\lamast) = \T(I,\xast,\ldots,\xast) - \lamast \xast = 0,
\eeq
where $\lamast = \mu(\xast)$ is such that ${\|\xast\|} = 1$.
This is precisely Equation \eqref{eq:fixed_point_def}.
%
For future use, we denote the {\em gradient} of $L(\xv,\lambda)$ at an arbitrary point $\xv \in S_{n-1}$
by
\beqn
\label{eq:gradient}
\grad(\xv) = \frac{1}{m}\left.\nabla_{\xv} L(\xv,\lambda)\right|_{\lambda = \mu(\xv)} = \T(I,\xv,\ldots,\xv) - \mu(\xv) \xv.
\eeqn
Similarly, we denote the {\em Hessian matrix} by
\begin{equation}
\label{eq:hessian}
H(\xv) = \frac{1}{m}\left.\nabla_{\xv}^2 L(\xv,\lambda)\right|_{\lambda = \mu(\xv)} =  (m-1)\T(I,I,\xv,\ldots,\xv) - \mu(\xv) I.
\end{equation}
As will become clear in Section \ref{sec:newton_based}, the spectral structure of $H(\xast)$ plays a fundamental role in the convergence of our proposed Newton-based methods to 
$(\xast,\lamast)$.

\section{Power methods for computing tensor eigenpairs}\label{sec:high_order}
To motivate our approach, it is first instructive to briefly review previous iterative methods,
specifically the symmetric higher-order power method (HOPM)
\cite{Lathauwer_2000} and the shifted-HOPM
\cite{kolda2011shifted, kolda2014adaptive}.
The HOPM was derived as a way to compute the best rank-1 approximation of a symmetric tensor under the squared error loss,
\begin{equation}
\argmin_{\xv\in\R^n}\nrm{ \mathcal T-\xv \otimes \ldots \otimes \xv}_F^2 = \argmin_{\xv\in\R^n} \sum_{\im=1}^n (t_{\im}-\xim)^2.
\end{equation}
Although the above problem is non-convex and has no closed form solution, it was shown in \cite{Lathauwer_2000} that it is equivalent to finding the vector $\xast$ with $\nrm{\xast}=1$ that maximizes the objective function $\mu(\xv)$ in \eqref{eq:induced_function}.
To compute $\xast$, the following generalization of the matrix power method to high-order tensors was proposed. Starting from a (random) initial  guess $\xv_{(0)}\in S_{n-1}$, HOPM iterates
\beqn
\label{eq:S-HOPM_iterate}
\xv_{(k+1)} = \frac{\nabla_{\xv} \mu(\xk)}{{\|\nabla_{\xv} \mu(\xk)\|}} = \frac{\T(I,\xk,\ldots,\xk)}{{\|\T(I,\xk,\ldots,\xk)\|}}.
\eeqn
It was shown in \cite[Theorem 4]{kofidis2002best} that for an even-order tensor, if its associated function $\mu(\xv)$ is {\em convex} or {\em concave}, then HOPM is guaranteed to converge to a local optimum of $\mu(\xv)$ in $S_{n-1}$. 
For general symmetric tensors, however, HOPM has no convergence guarantees, and may indeed fail to converge, see \cite{kofidis2002best} for a specific example.

To overcome the limitations of HOPM, \cite{kolda2011shifted,kolda2014adaptive} proposed the shifted function  
\begin{equation}
\mu_\alpha(\xv) = 
\mu(\xv)
+\alpha \nrm{\xv}^{m},
\qquad \alpha\in\R.
\end{equation}
Since on the unit sphere $\mu_\alpha(\xv)=\mu(\xv)+\alpha$, the critical points of $\mu$ and $\mu_\alpha$ are identical.
Instead of  \eqref{eq:S-HOPM_iterate}, the shifted-HOPM iterates
\beq
\xv_{(k+1)} = \frac{\nabla_{\xv} \mu_\alpha(\xk)}{{\|\nabla_{\xv} \mu_\alpha(\xk)\|}} = \frac{\T(I,\xk,\ldots,\xk)+\alpha \xk}{{\|\T(I,\xk,\ldots,\xk)+\alpha \xk\|}}.
\eeq
Importantly, the value of 
$\alpha$ can be tuned so that from any starting point $\xz \in S_{n-1}$, the shifted-HOPM is guaranteed to converge to a critical point of $\mu_\alpha$.
\cite{kolda2014adaptive} further devised an adaptive shifted-HOPM, whereby the value of $\alpha_{(k)}$ is updated at each iteration so that $\mu_{\alpha_{(k)}}(\xv)$ is locally convex or concave around $\xk$.
This avoids the possible slowdown of the shifted-HOPM with a fixed value of $\alpha$, while maintaining its convergence guarantees.
\begin{figure}
	\includegraphics[width = 0.45\textwidth]{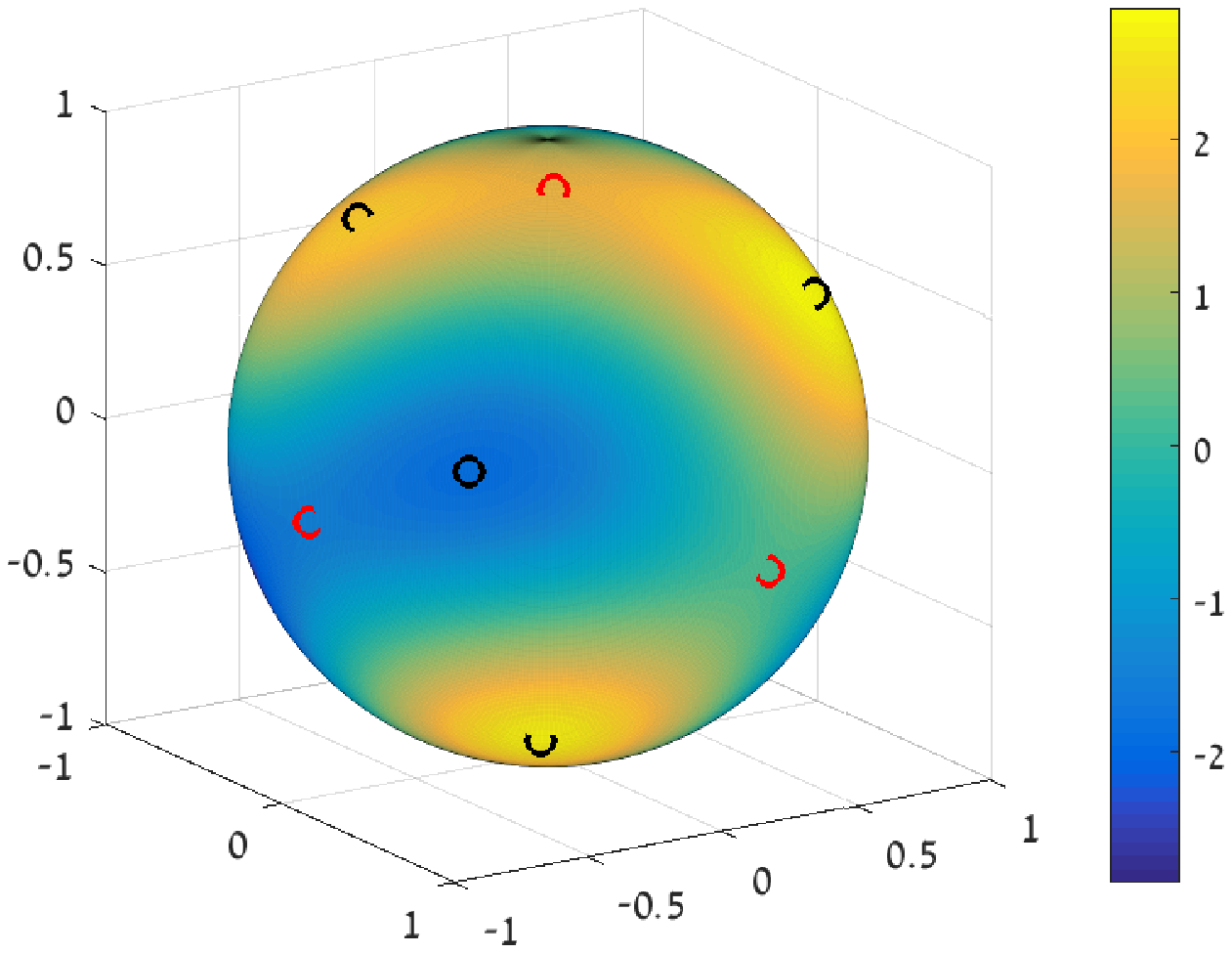}
	\hfill
	\includegraphics[width = 0.45\textwidth]{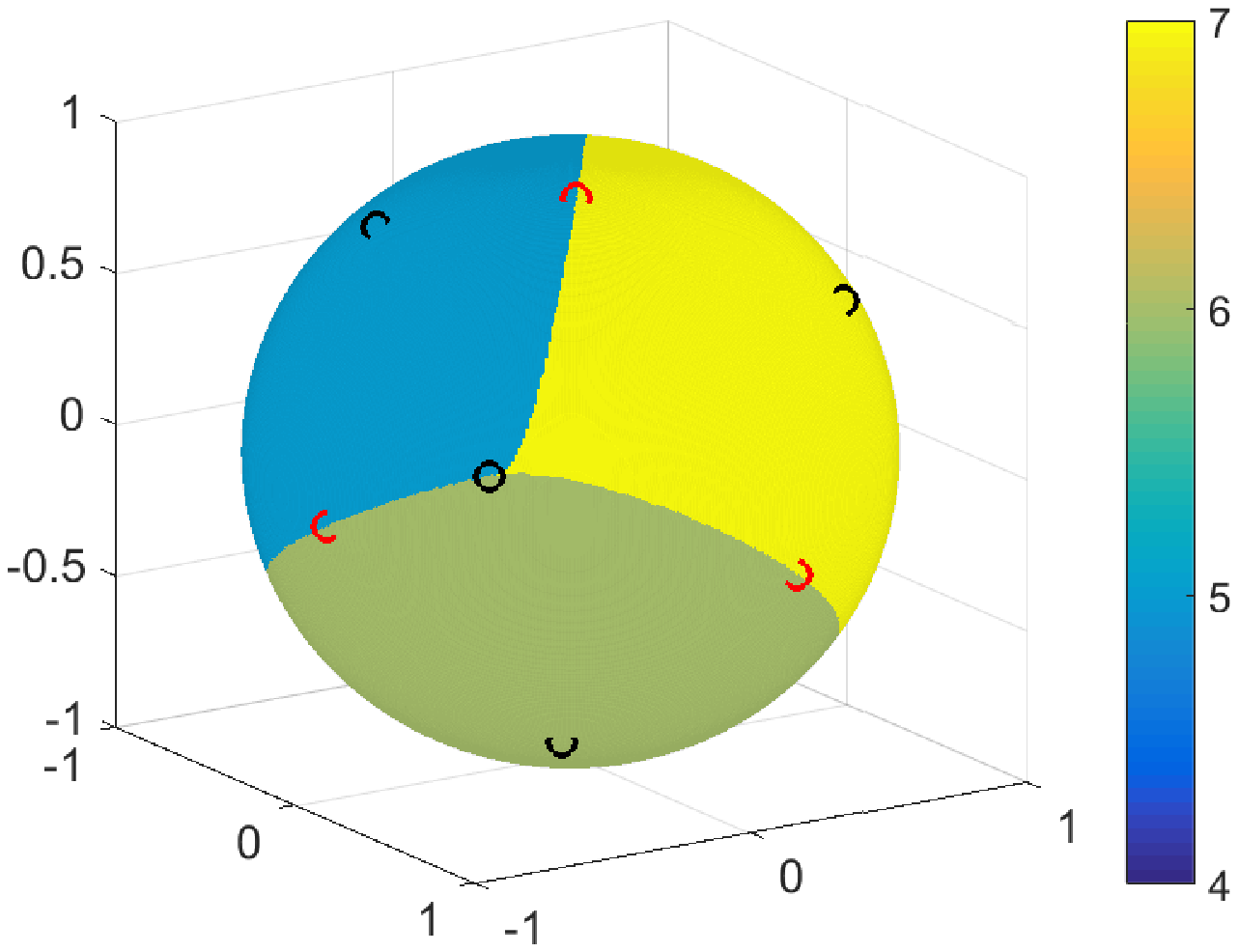}
	\caption{ Left: The value of $\mu(\xv)$ on the unit sphere for a random tensor of order $m=3$ and dimension $n=3$.	
The black/red circles indicate power-stable/power-unstable eigenvectors.
Right: Attracting regions for the adaptive HOPM.
The color of a point represents the eigenvector to which the method converged. 
}
	\label{fig:cp_power_methods}    
\end{figure}
\begin{figure}
	\includegraphics[width = 0.45\textwidth]{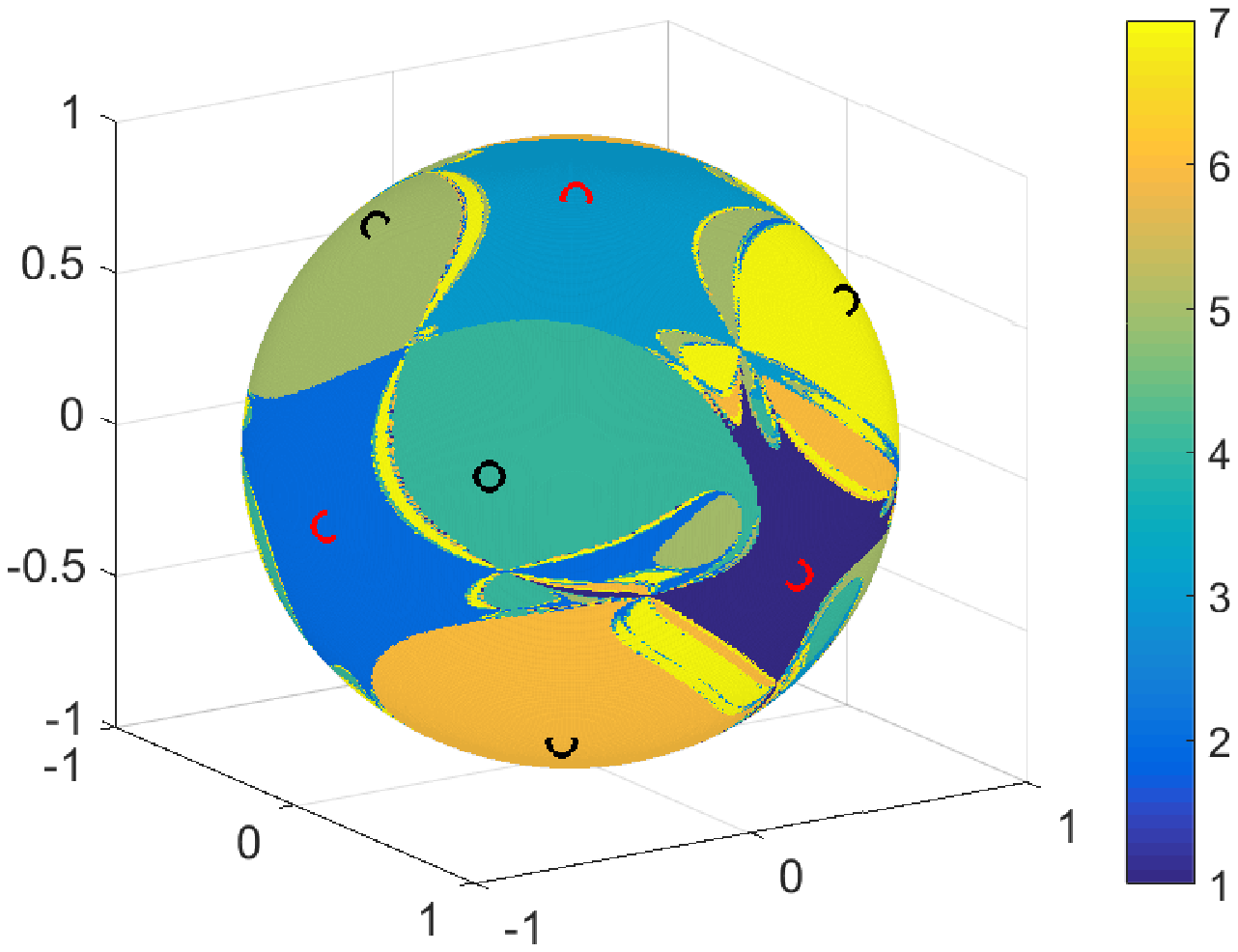}
	\hfill
	\includegraphics[width = 0.45\textwidth]{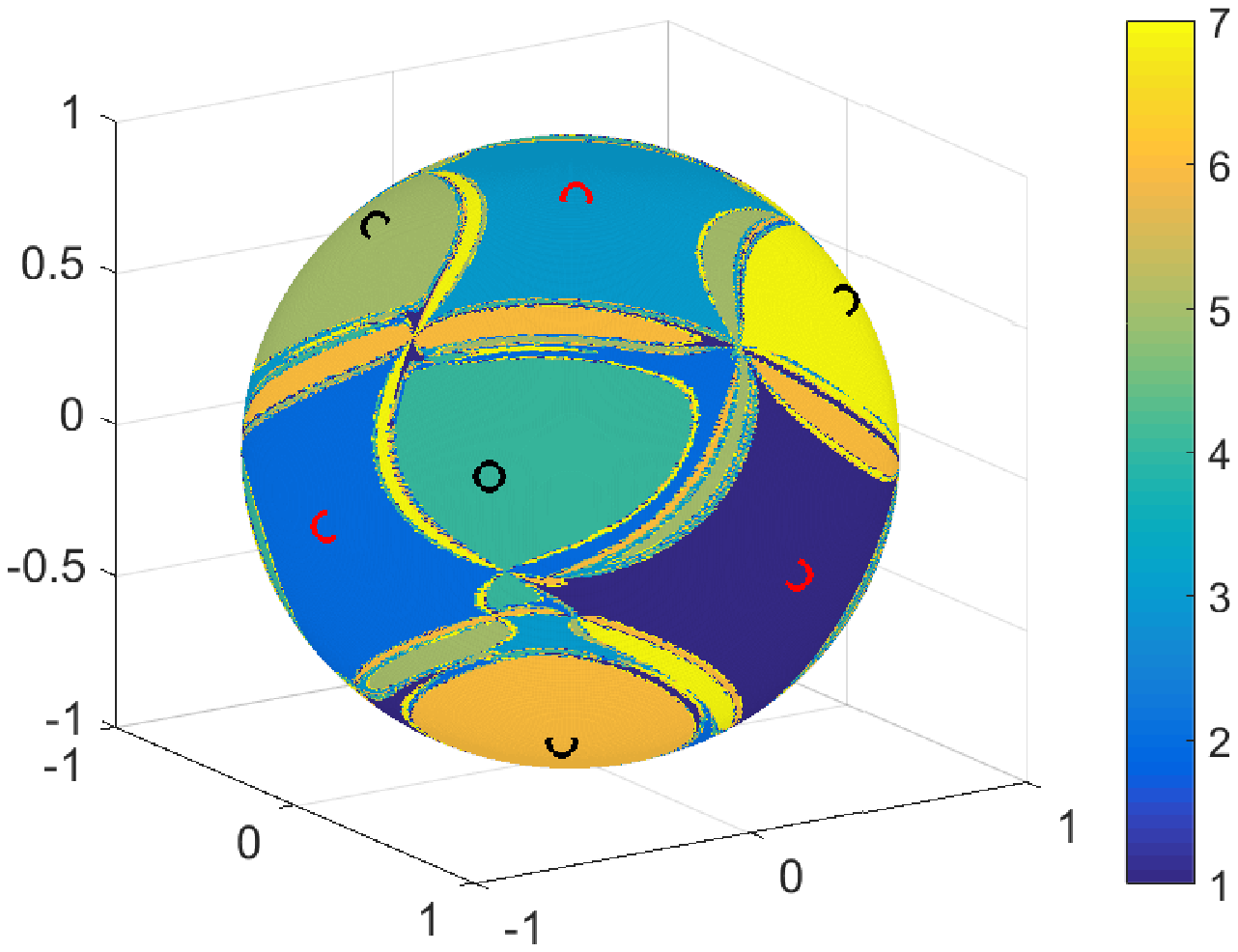}
	\caption{ Attracting regions for NCM (left) and O--NCM (right) for the same tensor as in Fig.\ref{fig:cp_power_methods}.
	}
	\label{fig:newton_attraction}    
\end{figure}

\paragraph{Convergence, attraction regions, and stable eigenpairs}
The adaptive shifted-HOPM converges only to some eigenpairs of a tensor.
These may be characterized as follows.
For any $\xv\in S_{n-1}$, let $U_{\xv} \in \R^{n \times n-1}$ be a matrix with $n-1$ orthonormal columns that span the subspace orthogonal to $\xv$.
Define the \textit{projected Hessian matrix},
\begin{equation}\label{eq:projected_hessian}
	H_p(\xv) =  U_{\xv}^T H(\xv)U_{\xv} \in \R^{(n-1) \times (n-1)},
\end{equation}
where $H(\xv)\in\R^{n\times n}$ is the Hessian matrix in \eqref{eq:hessian}. In \cite{kolda2011shifted}, an eigenvector $\xv^\ast$ was termed \textit{positive-stable} if $H_p(\xast)$ is positive-definite
and \textit{negative-stable} if $H_p(\xast)$ is negative-definite. Otherwise, $\xast$ is termed {\em unstable}.
\cite{kolda2011shifted} showed that the shifted HOPM does not converge to unstable eigenvectors but does converge to the stable ones. Further, the convergence is at a linear rate. 
To distinguish between eigenpairs that are stable for the (adaptive) shifted-HOPM
and those that are stable for the Newton-based methods, we henceforth refer
to the above as power-stable eigenpairs, power-unstable eigenpairs, etc.

As an example, the left panel of Figure \ref{fig:cp_power_methods} shows the value of $\mu(\xv)$ over the unit sphere for a $3 \times 3 \times 3$ symmetric tensor with $7$ real eigenvectors.
Three eigenvectors, depicted in red, are power-unstable, while the remaining four, depicted in black, are power-stable. The right panel of the same figure shows the results of the adaptive shifted-HOPM. The color indicates the eigenvector to which the method converged, starting from various locations on the unit sphere.
The figure shows clear convergence regions around $3$ of the power-stable eigenpairs.
The region around the fourth power-stable eigenpair appears on the back side of the sphere.
In agreement with theory, the adaptive HOPM did not converge to any of the three power-unstable eigenpairs. 

\section{Newton-based methods for the tensor eigen-problem}\label{sec:newton_based}
Given the limitations of the aforementioned methods,
our goal is to derive a fast iterative algorithm that under mild assumptions is able to converge to \textit{all}
real eigenpairs of a symmetric tensor.
To this end, we develop a Newton-based method.

\subsection{Newton correction method}\label{sec:newton}
Several variants of Newton's method were derived for the symmetric matrix eigen-problem, see for example \cite[Chapter 6]{stewart2001matrix}.
Here we derive a Newton-based method
for computing the eigenpairs of symmetric tensors.
Recently, \cite{guo2017modified} considered a similar approach for finding some nonnegative eigenpairs of a nonnegative tensor.

Let $(\xast,\lamast)$ be an eigenpair of a symmetric tensor $\T$ of order $m$ and dimensionality $n$.
Given an approximation $\xv$ to $\xast$, our goal is to obtain an improved approximation $\xv'$ (see Figure \ref{fig:ONCM_def}, left).
Denote the {\em exact} unknown correction by $\yast = \xast - \xv$ and
recall $\mu(\xv)=\T(\xv,\ldots,\xv)$. Since $\xast = \xv + \yast$ and  $\lamast = \mu(\xast) = \mu(\xv + \yast)$, the eigen-problem in (\ref{eq:fixed_point_def}) can be written as
\beqn
\label{eq:ev_full}
\T(I, \xv + \yast,\ldots, \xv + \yast) = \mu(\xv + \yast) \cdot (\xv + \yast).
\eeqn
Recalling the Hessian matrix $H(\xv)$ in \eqref{eq:hessian}, we define the matrix $A(\xv) \in \R^{n \times n}$ by
\begin{equation}
\label{eq:A_def}
                A(\xv)=  H(\xv) - m \xv \T(I,\xv,\ldots,\xv)^T.
\end{equation}
Setting apart the terms that are linear in $\yast$,
Eq. \eqref{eq:ev_full} takes the form
\beqn
\label{eq:system_quad_simple}
A(\xv) \yast = -\grad(\xv) + \Delta(\xv,\yast),
\eeqn
where $\grad(\xv)$ is given in \eqref{eq:gradient}.
Here, $\Delta(\xv,\yast)$ accounts for all high order terms in $\yast$,
\begin{flalign}
\nonumber
\Delta(\xv,\yast) &= \sum_{i=2}^m \binom{m}{i}\T(\underbrace{\xv,\ldots,\xv}_{m-i},\underbrace{\yast,\ldots,\yast}_{i})\xv + \sum_{i=1}^m \binom{m}{i}\T(\underbrace{\xv,\ldots,\xv}_{m-i},\underbrace{\yast,\ldots,\yast}_{i})\yast
\\&- \sum_{i=2}^{m-1} \binom{m-1}{i}\T(I,\underbrace{\xv,\ldots,\xv}_{m-i-1},\underbrace{\yast,\ldots,\yast}_{i}).
\label{eq:simple_high_order_terms}
\end{flalign}

\begin{algorithm}[t]   
        \caption{Newton correction method}\label{algo:newton_correction}   
        \begin{algorithmic}[1]
                \STATE Input: tensor $\T\in\R^{n\times\dots\times n}$, tolerance parameter $\delta>0$
                \STATE Initialization: Randomly choose $\xz \in S_{n-1}$
                \STATE Set $k=0$ and $\lambda_{(0)} = \mu(\xz)$ 
                \WHILE {$\|\xv_{(k)}-\xv_{(k-1)}\|>\delta$}
                \STATE Compute $\yk = -A(\xk)^{-1}\bm g(\xk)$
                \STATE Set $\xkpo = (\xk+\yk)/{\|\xk+\yk \|}$ 
                \STATE Set $\lambda_{(k+1)} = \mu(\xkpo)$       
                \STATE $k \leftarrow k+1$
                \ENDWHILE      
                \RETURN $(\xk,\lambda_{(k)})$      
\end{algorithmic}\end{algorithm}

By definition, the solution $\yast$ to \eqref{eq:system_quad_simple} satisfies $\xv + \yast = \xast$.
However, solving \eqref{eq:system_quad_simple} exactly for $\yast$ is as difficult as finding the eigenpair $(\xast,\lamast$) of the tensor $\T$ we started from.
Instead, we devise an iterative Newton correction method (NCM) that solves \eqref{eq:system_quad_simple} only approximately.
Given the approximation $\xk$ of $\xast$ at the $k^{\text{th}}$ iteration, NCM computes a new approximation $\xkpo$ by neglecting the high order terms $\Delta(\xv,\yast)$ in \eqref{eq:system_quad_simple}. 
This amounts to solving the system of $n$ linear equations
\begin{equation}\label{eq:system_lin_simple}
A(\xk) \yk = -\grad(\xk).
\end{equation} 
Assuming $A(\xk)$ is invertible, the unique solution to \eqref{eq:system_lin_simple} is given by
\beqn
\label{eq:system_lin_simple_solution}
\yk = -A(\xk)^{-1}\grad(\xk).
\eeqn
The new approximation $\xkpo$ for $\xast$ is then
\beqn
\label{eq:normalized_update}
\xkpo = \frac{\xk +\yk}{{\|\xk +\yk\|}}.
\eeqn
Given an initial guess $\xz\in S_{n-1}$, NCM iterates steps \eqref{eq:system_lin_simple_solution} and \eqref{eq:normalized_update}.
Once a stopping condition is met, the pair $(\xk,\mu(\xk))$ is returned;
see Algorithm \ref{algo:newton_correction}.

The left panel of Figure \ref{fig:newton_attraction} shows the convergence regions of NCM for the eigenpairs of the same tensor as in Figure \ref{fig:cp_power_methods}.
In this case, all eigenpairs are attracting points of NCM and can thus be found by running  Algorithm \ref{algo:newton_correction} multiple times with different (random) initial guesses.

\subsubsection*{Convergence guarantees}
\label{sec:convergence}
Two questions regarding NCM are (i) to which eigenpairs of $\T$ the method can converge to? and (ii) what is the convergence rate?
To answer these questions we recall the definition of the Hessian matrix in (\ref{eq:hessian}). Given an eigenpair $(\xast,\lamast)$, its corresponding Hessian matrix 
is
\[
H(\xast) = (m-1)\T(I,I,\xast,\ldots,\xast)-\lamast I.
\]
Note that $H(\xast)$ is symmetric and has $\xast$ as an eigenvector with eigenvalue $(m-2)\lamast$. 
Denote by $\mu_1^\ast, \ldots, \mu_{n-1}^\ast$ the other $n-1$ eigenvalues of $H(\xast)$.
By definition, these are the eigenvalues of the projected Hessian $H_p(\xast)$ in \eqref{eq:projected_hessian}.

\begin{definition}
\label{def:g-stable}
For $\g>0$, an eigenpair $(\xast,\lamast)$ is \emph{$\g$-{Newton-stable}} if all eigenvalues of $H_p(\xast)$ in absolute value are at least $\g$, namely, $\min_{i}|\mu_{i}^*| \geq \g$.
\end{definition}

Note that $H_p(\xast)$ is full rank iff $\xast$ is $\g$-Newton-stable for some $\gamma>0$.
Similarly, $H(\xast)$ is full rank iff $\xast$ is $\g$-Newton-stable for some $\g>0$ and $\lamast \neq 0$.
We have the following convergence guarantee for NCM. The proof is given in Appendix \ref{sec:NCM_proofs}.

\begin{theorem}
\label{lem:simple_converge}
Let $(\xast,\lamast)$ be an eigenpair of a symmetric tensor $\T$. Suppose that $(\xast,\lamast)$ is $\g$-Newton-stable and that $\lamast \neq 0$.
Then there exists an $\varepsilon = \varepsilon(\g,\lamast) >0$
 such that for any $\xz$ that satisfies $\nrm{\xz-\xast}<\varepsilon$, the sequence $\xz, \xv_{(1)}, \dots$, computed by Algorithm \ref{algo:newton_correction} converges to $\xast$ at a quadratic rate.
\end{theorem}

\subsection{Orthogonal Newton correction method} \label{sec:newton_orthogonal}
\begin{figure}%
\centering
\hspace{1.2cm}
\begin{tikzpicture}
  [
    scale=2.8,
    >=stealth,
    point/.style = {draw, circle,  fill = black, inner sep = 1pt},
    dot/.style   = {draw, circle,  fill = black, inner sep = .2pt},
    brace/.style={decoration={brace, mirror, raise=3pt}, decorate}
  ]

  \def\rad{1}
  \node (origin) at (0,0) [point]{};	    
  \draw [thick,domain=-0:90] plot ({cos(\x)}, {sin(\x)});  

  \node (n1) at +(65:\rad) [point, label = above:$\xv$] {};
  \node (n2) at +(15:\rad) [point] {};  
  \node (n3) at +(30:\rad) [point] {};
  \node (u) at ($ (origin) ! 1.5 ! (n3) $) [point] {};
  \node  at ([shift={(-5:.4)}] n3) {$\xv' =\frac{\xv+\yv}{\nrm{\xv+\yv}}$};
  \node  at ([shift={(-0:.13)}] n2) {$\xast$};
  \draw[->] (origin) --  (n1);
  \draw[->] (origin) --  (n2);
  \draw[->] (origin) --  (n3);  
  \draw[->] (n1) -- node [ left = 3pt] {$\yast$} (n2);
  \draw[->] (n1) -- node [above right] {$\yv$} (u);
  \draw[dotted] (origin) -- (u);
\path  (origin) --  node [below = 4pt] {} (1,0);


\end{tikzpicture}
\hspace{1.5cm}
\begin{tikzpicture}
  [
    scale=2.8,
    >=stealth,
    point/.style = {draw, circle,  fill = black, inner sep = 1pt},
    dot/.style   = {draw, circle,  fill = black, inner sep = .2pt},
    brace/.style={decoration={brace, mirror, raise=3pt}, decorate}
  ]

  \def\rad{1}
  \node (origin) at (0,0) [point]{};
  \draw [thick,domain=-0:90] plot ({cos(\x)}, {sin(\x)});  

  \node (n1) at +(65:\rad) [point, label = above:$\xv$] {};
  \node (n2) at +(15:\rad) [point, label = right:$\xast$] {};  
  \node (n3) at +(30:\rad) [point] {};
  \node (u) at ($ (origin) ! 1.21 ! (n3) $) [point] {};
  \node (uast) at ($ (origin) ! .66 ! (n2) $) [point] {};
  \node  at ([shift={(-7:.4)}] n3) {$\xv' = \frac{\xv+\bm u}{\nrm{\xv+\bm u}}$};
  \draw[->] (origin) --  (n1);
  \draw[->] (origin) --  (n2);
  \draw[->] (origin) --  (n3);  
  \draw[->] (n1) -- node [ left = 0pt] {$\bm u^*$} (uast);
  \draw[->] (n1) -- node [above right] {$\bm u$} (u);
  \draw[dotted] (origin) -- (u);


\draw [brace] (origin) --  node [below = 4pt] {$\alpha$} (uast);

  \def\ralen{.5ex}  
  \foreach \inter/\first/\last in {uast/origin/n1, n1/origin/u}
    {
      \draw let \p1 = ($(\inter)!\ralen!(\first)$), 
                \p2 = ($(\inter)!\ralen!(\last)$),  
                \p3 = ($(\p1)+(\p2)-(\inter)$)      
            in
              (\p1) -- (\p3) -- (\p2)               
              ;
    }
\end{tikzpicture}
\caption{Illustration of one iteration of NCM (left) and O--NCM (right).}%
\label{fig:ONCM_def}%
\end{figure}
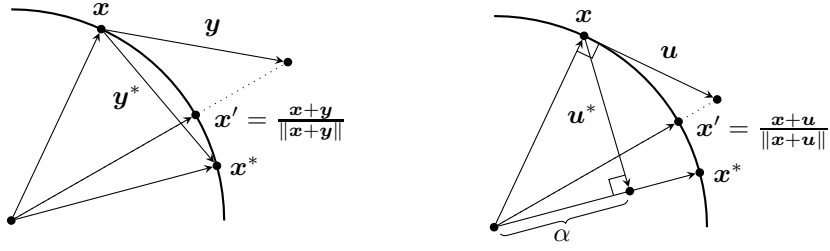
As discussed above, the NCM method may not converge to an eigenvector with eigenvalue $\lamast=0$. To remove this limitation, we now develop an orthogonal NCM variant.
Given an approximation $\xv\in S_{n-1}$ of $\xast$, we first decompose it into
its projection onto $\xast$ and a residual (see Figure \ref{fig:ONCM_def}, right),
\begin{equation}\label{eq:theta_k}
\xv = \alpha \xast - \ortc^*,
\end{equation}
where $\alpha = \xv^T \xast$ and $\ortc^* = \alpha \xast - \xv$ is the residual.
Since $\xast$ and $\ortc^*$ are orthogonal and $\xv\in S_{n-1}$, $\nrm{\xv}^2=\alpha^2+\nrm{\ortc^*}^2 =1$.
For reasons to become clear shortly, 
we also introduce a correction $\beta^*\equiv\beta^*(\xv,\xast)$ to the eigenvalue $\lambda^{\ast}$, defined as
\beqn
\label{eq:beta_ast}
\beta^* =  \alpha^{m-2}\lambda^{\ast} - \T(\xv,\dots,\xv)= \alpha^{m-2}\lambda^{\ast} - \mu(\xv).
\eeqn
When $\xv = \xast$, we have $\ortcast =0$, $\alpha = 1$ and $\beta^* =0$.
Since $(\xast,\lamast)$ is an eigenpair,
\[
\T(I,\alpha \xast,\ldots,\alpha\xast) = 
\alpha^{m-1}\lambda^\ast \xast.
\]
Inserting $\xast = \frac{1}{\alpha}(\xv + \ortc^*)$ and 
$\beta^*$ 
into the above equation gives
\begin{align}
\begin{split}
\T(I,\xv +\ortc^*,\ldots,\xv +\ortc^*) = (\mu(\xv) +\beta^*)\cdot (\xv +\ortc^*).
\end{split}
\end{align}
We set apart the terms that are linear in $\ortc^*$ and $\beta^*$ to obtain
\newcommand{\tD}{\tilde \Delta}
\begin{equation}\label{eq:lin_exact_part}
H(\xv)\ortc^* -\beta^* \xv =  
-\grad(\xv)
+ \tD(\xv,\ortc^*, \beta^*),
\end{equation}
where $\grad$ and $H$ were defined in (\ref{eq:gradient}) and (\ref{eq:hessian}) respectively and $\tD(\xv,\ortc^*, \beta^*)$ includes all the remaining higher order terms in $(\ortc^*,\beta^*)$,
\beqn
\label{eq:Delta_ONCM}
\tD(\xv, \ortc^*, \beta^*) =  \beta^* \ortc^*  - \sum_{i=2}^{m-1} \tbinom{m-1}{i}\T(I,\underbrace{\xv,\ldots,\xv}_{m-i-1},\underbrace{\ortc^*,\ldots,\ortc^*}_{i}).
\eeqn
Combining \eqref{eq:lin_exact_part} with the orthogonality condition, \( (\xast)^T \ortc^* \propto (\xv + \ortc^*)^T \ortc^* = 0\), gives the following set of non-linear equations in $(\ortc^\ast,\beta^*)$,
\beqn
\label{eq:system_quad}
{\begin{pmatrix}
H(\xv) & - \xv \\
\xv^T & 0 \\
\end{pmatrix}
\begin{pmatrix}
\ortc^* \\ \beta^*
\end{pmatrix}} =
-{\begin{pmatrix}
\grad(\xv) \\ 0
\end{pmatrix}
+
\begin{pmatrix}
\tD(\xv,\ortc^*, \beta^*)
  \\ -\nrm{\ortc^*}^2
\end{pmatrix}}.
\eeqn
By construction, the solution $(\ortc^\ast,\beta^*)$ to \eqref{eq:system_quad} satisfies $(\xv +\ortc^\ast)/\nrm{\xv +\ortc^\ast}  = \xast$. 
Similarly to the NCM, we neglect the high order terms in the right hand side of (\ref{eq:system_quad}) and solve the system of linear equations
in the $n+1$ unknowns $(\ortc,\beta)$,
\begin{equation}\label{eq:system_lin}
\begin{pmatrix}
H(\xv) & - \xv \\
\xv^T & 0 \\
\end{pmatrix}
\begin{pmatrix}
\ortc \\ \beta
\end{pmatrix}=
-\begin{pmatrix}
\grad(\xv) \\ 0
\end{pmatrix}.
\end{equation} 

\begin{algorithm}[t]   
        \caption{Orthogonal Newton correction method}\label{algo:newton_orthogonal}   
        \begin{algorithmic}[1]
                \STATE Input: tensor $\T$, threshold $\delta$
                \STATE Initialization: Randomly choose $\xz \in S_{n-1}$
                \STATE Set $k=0$ and $\lambda_{(0)} = \mu(\xz)$
                \WHILE {$\|\xv_{(k)}-\xv_{(k-1)}\|>\delta$}
                \STATE Compute $\ortc_{(k)}= - U_{\xk} H_p(\xk)^{-1} U_{\xk}^T \bm \grad(\xk)$
                \STATE Set $\xkpo = (\xk+\ortc_{(k)})/\nrm{\xk+\ortc_{(k)}}$
                \STATE Set $\lambda_{(k+1)} = \mu(\xkpo)$
                \STATE $k \leftarrow k+1$
                \ENDWHILE      
                \RETURN $(\xk,\lambda_{(k)})$      
\end{algorithmic}\end{algorithm}

Due to the extra variable $\beta$, \eqref{eq:system_lin} seems to be of dimension $n+1$, as opposed to the $n$ dimensional system in \eqref{eq:system_lin_simple}. However, as we now show, the system in \eqref{eq:system_lin} can be equivalently solved in the $n-1$ dimensional subspace orthogonal to $\xv$.
More precisely, 
let $P_{\xv}^\perp = (I-\xv \xv^T)$ be the projection matrix into the subspace orthogonal to $\xv$
and let 
$U_{\xv} \in \R^{n \times (n-1)}$ have orthonormal columns
such that $P_{\xv}^\perp = U_{\xv} U_{\xv}^T$.
Recall the projected Hessian matrix $H_p(\xv) = U_{\xv}^T \tA(\xv) U_{\xv}$ in \eqref{eq:projected_hessian}.
The following lemma is an adaptation of \cite[Theorem 6.2.2]{stewart2001matrix} to our  setting.
Its proof is given in Appendix \ref{sec:ONCM_proofs}.

\begin{lemma}
\label{lem:stewart}
A vector $\ortc \in\R^n$ satisfies \eqref{eq:system_lin} if and only if $\bz\in\R^{n-1}$ satisfies
\beqn
\label{eq:Cz_d}
H_p(\xv) \bz = - U_{\xv}^T \grad(\xv)
\quad \text{and}
\quad
\ortc = U_{\xv} \bz.
\eeqn
\end{lemma}


Assuming $H_p(\xv)$ is invertible,
the solution to \eqref{eq:Cz_d} is 
\(
\ortc = -U_{\xv} H_p(\xv)^{-1} U_{\xv}^T \grad(\xv).
\)
%
So given the $k^{\text{th}}$ approximation $\xk$ to $\xast$, 
O--NCM computes
\begin{equation}\label{eq:ortho_update}
\ortc_{(k)} = -U_{\xk} H_p(\xk)^{-1} U_{\xk}^T \grad(\xk)
\end{equation}
and the new approximation is $\xkpo = (\xk +\ortc_{(k)})/\nrm{\xk+\ortc_{(k)}}$.
Given an initial $\xz$, O--NCM iterates these steps until a stopping condition is met; see Algorithm \ref{algo:newton_orthogonal}.

The right panel of Figure \ref{fig:newton_attraction} shows the convergence regions of O--NCM for the various eigenpairs of the same tensor in Figure \ref{fig:cp_power_methods}.
Similarly to NCM,
in this case, all eigenpairs are attracting points of O--NCM, but with slightly different regions.

\subsubsection*{Convergence guarantees}
We have the following convergence guarantee for O--NCM.
It is similar to that of NCM in Theorem \ref{lem:simple_converge}, but with the condition $\lambda^*\neq 0$ removed.
The proof is given in Appendix \ref{sec:ONCM_proofs_conv}.

\begin{theorem}
\label{lem:orthogonal_converge}
Let $(\xast,\lamast)$ be a $\g$-Newton-stable eigenpair of a symmetric tensor $\T$. 
There exists an $\varepsilon = \varepsilon(\g) >0$ such that for any $\xz\in S_{n-1}$ that satisfies $\nrm{\xz-\xast}<\varepsilon$, the sequence $\xz, \xv_{(1)}, \dots$, computed by Algorithm \ref{algo:newton_orthogonal} converges to $\xast$ at a quadratic rate.
\end{theorem}

\begin{remark}
	Recall that any eigenpair $(\xast,\lamast)$ to which the shifted-HOPM method
	converges to has a Hessian matrix which is either positive definite or negative definite. In either case, this Hessian matrix has full rank, and thus by Theorem
	\ref{lem:orthogonal_converge}, is a stable fixed point of O--NCM. In other words,
	O--NCM typically converges to many more tensor eigenpairs than the shifted power method. 
However, the adaptive shifted-HOPM is guaranteed to converge from {\em any} initial point, whereas no such global convergence guarantee is currently available for the Newton-based methods.  
\end{remark}
\newcommand{\weight}{\omega}
\section{Convergence to eigenpairs  with a rank deficient Hessian}
\label{sec:fail} 
The sufficient condition in Theorem
\ref{lem:orthogonal_converge} for O--NCM
to converge to an eigenpair $(\xast,\lamast)$ rests on the smallest absolute eigenvalue of 
the projected Hessian matrix $H_p(\xast)$.
In particular, if the eigenpair 
is $\gamma$-Newton-stable for some $\g>0$, an
attraction neighborhood around $\xast$ exists.
We now illustrate that this sufficient condition for convergence is by no means necessary. 
To this end, we analyze a simple example.
Denote $\bm 1 = (1,\ldots,1)^T\in\R^n$.
For any $\weight\in\R$, define the following cubic $n$-dimensional symmetric tensor,
\begin{equation}\label{eq:tensor_example}
\T_\weight = \sum_{i=1}^n \bm e_i \otimes \bm e_i \otimes \bm e_i + \weight(\bm 1 \otimes \bm 1 \otimes \bm 1).
\end{equation}
When $\weight=0$, $\T_\weight$ is orthogonal, having the maximally possible number of $2^n-1$ real eigenpairs. All these are Newton-stable and among them are the $n$ power-stable eigenpairs $\{(\bm e_i,1)\}_{i=1}^n$. 
Assume $n$ is odd and denote $l= \floor{n/2}$.
Let $N(\weight)$ be the number of real eigenpairs of $\T_\weight$.
The following proposition is proved in Appendix \ref{sec:example_proof}.

\begin{figure}
	\centering
	\includegraphics[width = 0.45\textwidth]{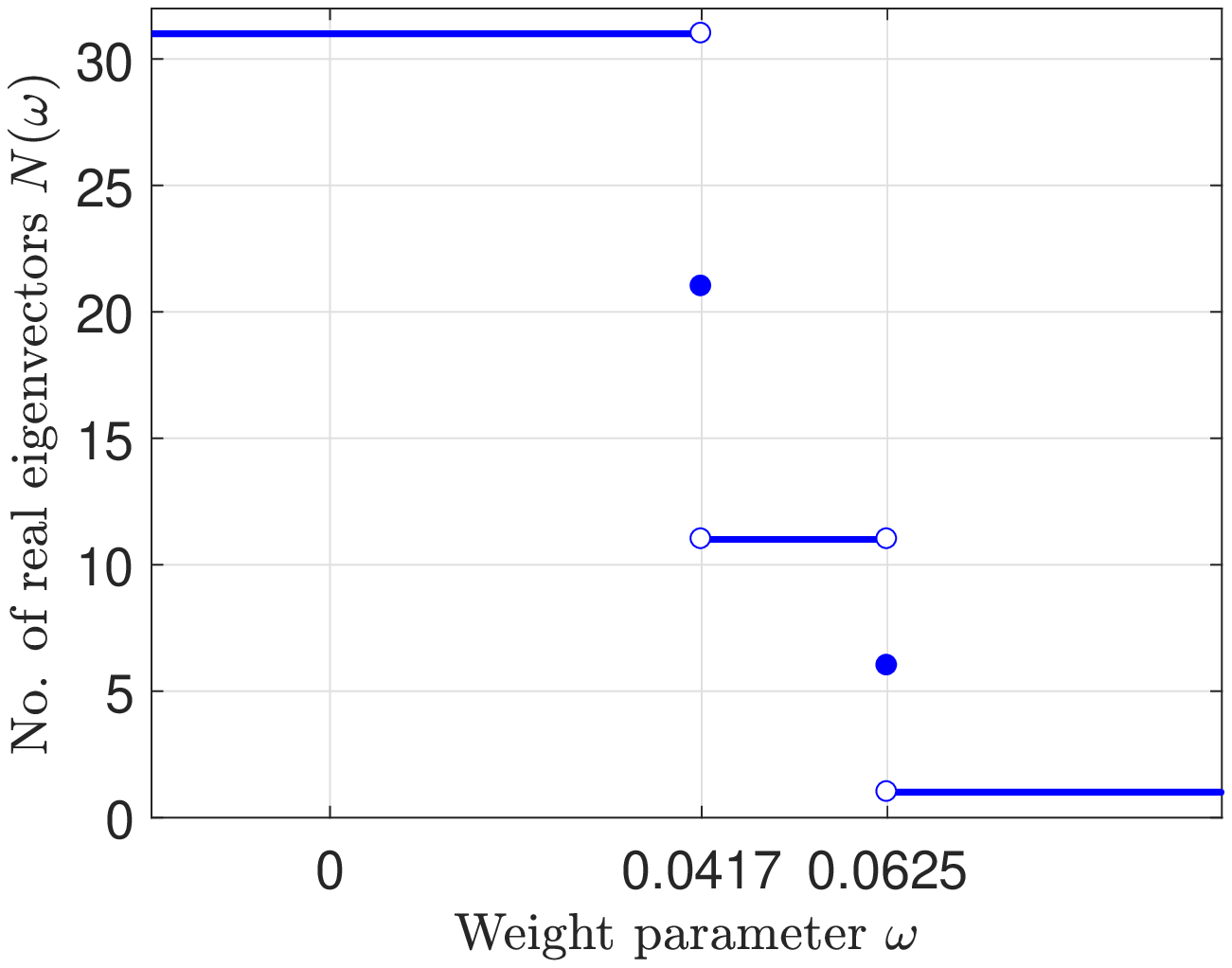}
\hfill
	\includegraphics[width = 0.45\textwidth]{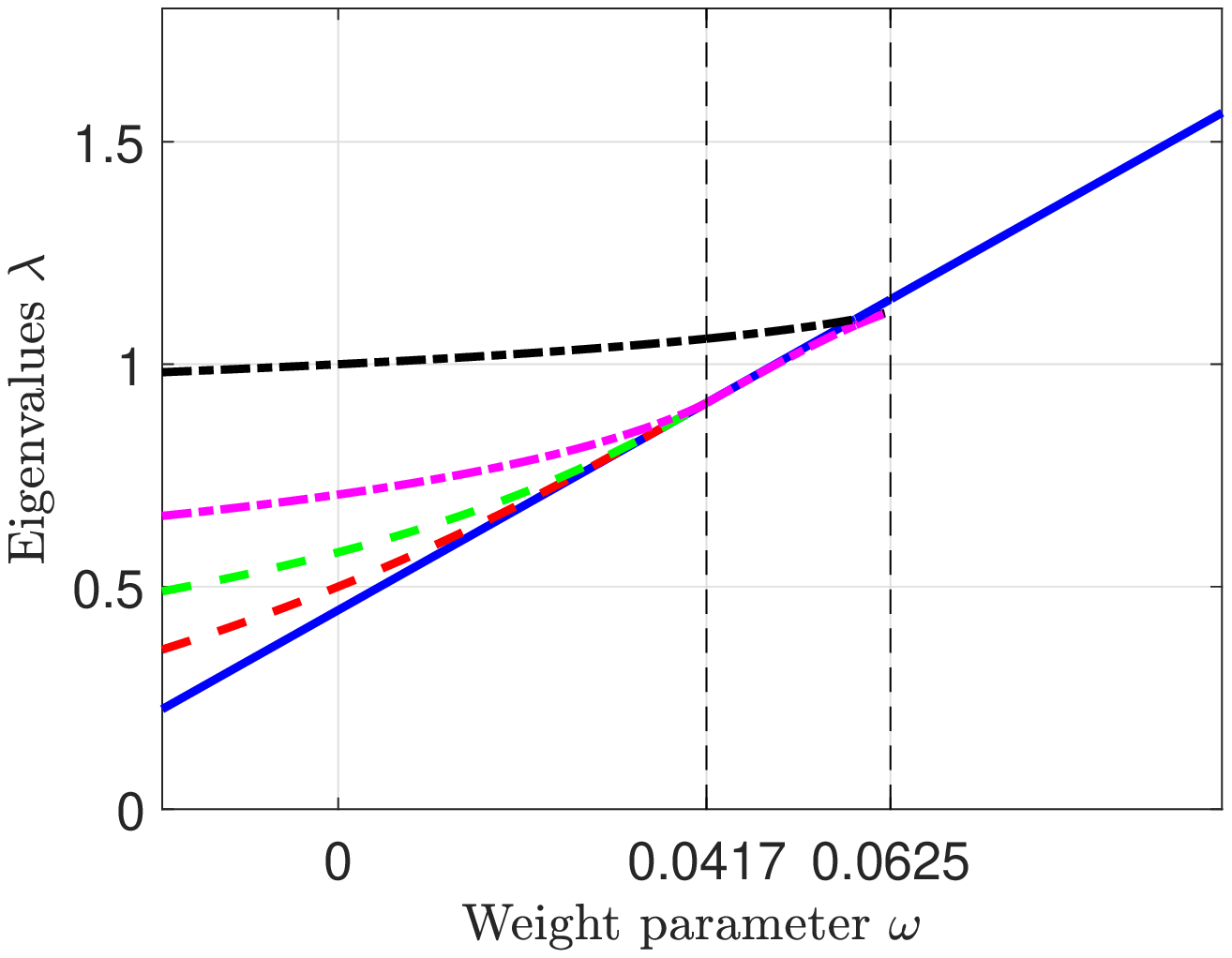}
	\caption{Number of real eigenpairs (left) and corresponding eigenvalues (right) for $\T_\weight$
with $n=5$. Although for small values of $\weight$ there are $31$ distinct eigenpairs, some of the eigenvalues are equal due to the tensor's symmetry.}	
	\label{fig:numer_of_real_eigenvalues}    
\end{figure}

\begin{proposition}\label{lem:number_real_eigenpairs}
Define $l$ \textit{thresholds},
$\weight_i = \frac{1}{4(l - i)(n-l+i)} > 0$,  $i \in \{0,\ldots,l-1\}$.
\begin{itemize}
\item[(i)] 
The number $N(\weight)$ of real eigenpairs of $\T_\weight$ of Eq. \eqref{eq:tensor_example} is
	\[
	N(\weight) =  \begin{cases}
	2^n-1 &  \weight < \weight_0\\
	1+2\sum_{j=1}^{l-i-1}\binom{n}{j} +\binom{n}{l-i}  & \weight = \weight_i \\
	1+\sum_{j=1}^{l-i-1} \binom{n}{j}  & \weight_{i}<\weight < \weight_{i+1} \\
	1 & \weight> \weight_{l-1}.
	\end{cases} 
	\]	
	\item[(ii)] 
For $\weight_i\in\{\weight_0,\dots,\weight_{l-1}\}$, $\binom{n}{l-i}$ out of the $N(\weight_i)$ real eigenpairs of $\T(\weight_i)$ are not \textit{Newton-stable}.
\end{itemize}
\end{proposition}

We illustrate the above properties for $\T_\weight$ with 
$n=5$.
In this case,  $l = \floor{n/2}=2$ and there are two thresholds, $\weight_0 = \frac{1}{4\cdot 2\cdot 3} \approx 0.0417$ and $\weight_1 = \frac{1}{4\cdot 1\cdot 4}=0.0625$.
Figure \ref{fig:numer_of_real_eigenvalues} shows the number $N(\weight)$ of real eigenpairs (left), and the different eigenvalues of $\T_\weight$ (right) as computed by O--NCM.
As expected, at $\weight_0$ and $\weight_1$, the number of real eigenvalues decreases.


Next, we examine the convergence of O--NCM on $\T_{\weight=\weight_0}$ with $n=3$. 
According to Proposition \ref{lem:number_real_eigenpairs}, $\weight_0=0.125$,
and the number of real eigenpairs is $N(\weight_0)=1+\binom{3}{1}=4$, three of which are not Newton-stable.
Figure \ref{fig:singular_point_attraction_b} shows the attraction regions around two of the eigevectors of $\T_{\weight_0}$, as well as the full unit sphere.
%
On the left, the eigenvector is Newton-stable. As expected, O--NCM converged to this eigenvector from any point in its neighborhood.
In contrast, the eigenvector on the right is not Newton-stable.
In this case, there is a positive probability of converging to a different eigenvector even when the initial guess is arbitrary close. 
Nonetheless, O--NCM converged to this eigenvector from {\em some} directions, even though the sufficient condition in Theorem \ref{lem:orthogonal_converge} does not hold.  

In this example, all eigenpairs of $\T_\weight$ are {\em isolated}, namely each one of them is the unique eigenpair in a small neighborhood around it.
In addition, the eigenvectors which are not Newton-stable have a projected Hessian matrix that is rank deficient but \textit{non-zero}.
In Appendix \ref{app:unstable}, we illustrate the behavior of O--NCM near eigenvectors that are either non-isolated, or have a projected Hessian matrix equals to zero. In these cases, O--NCM may not converge.


\begin{figure}
	\centering
	\includegraphics[width = 0.32\textwidth]{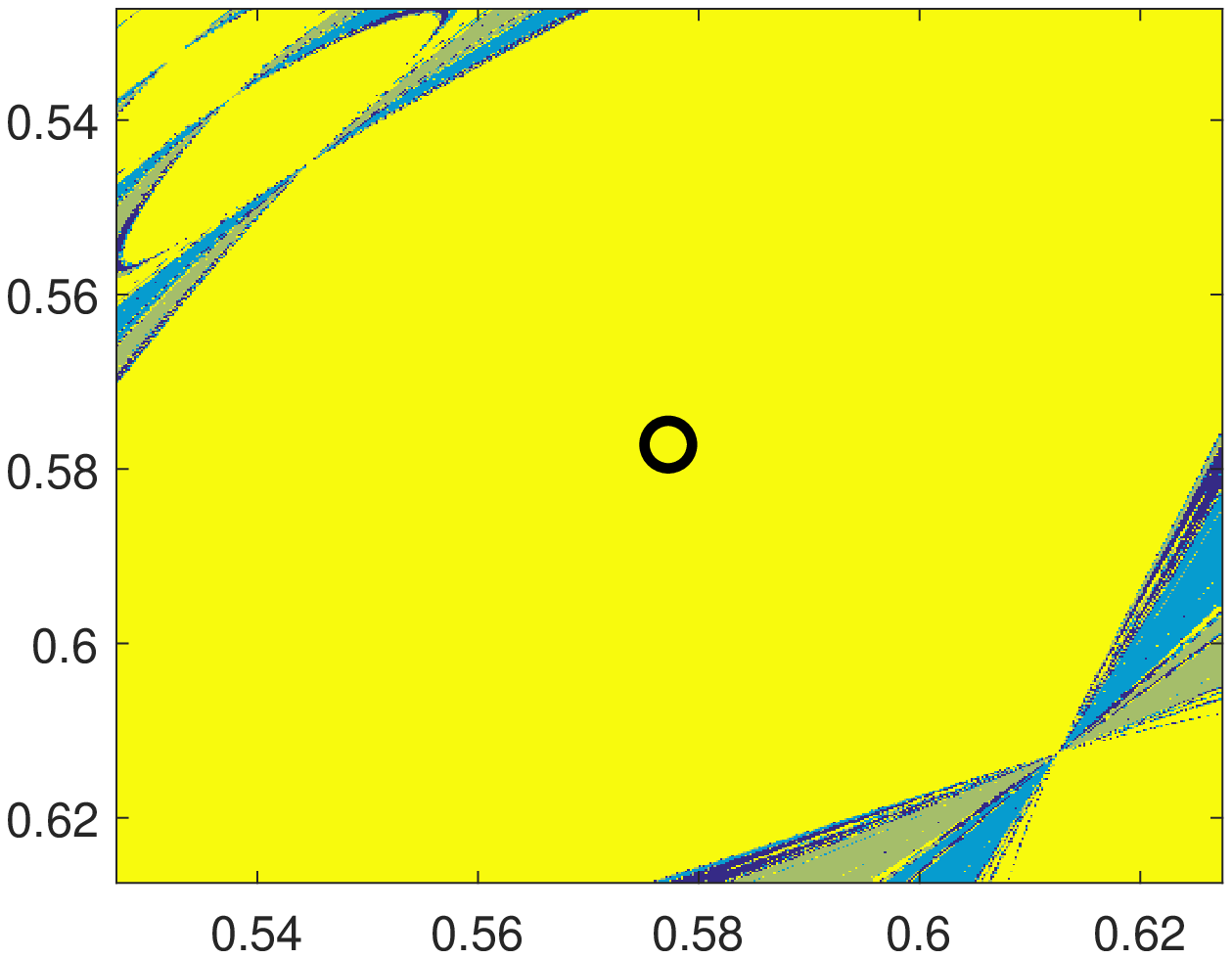}\hfill
	\includegraphics[trim={2.8cm 2cm 2.8cm 2.2cm},clip,width = 0.32\textwidth]{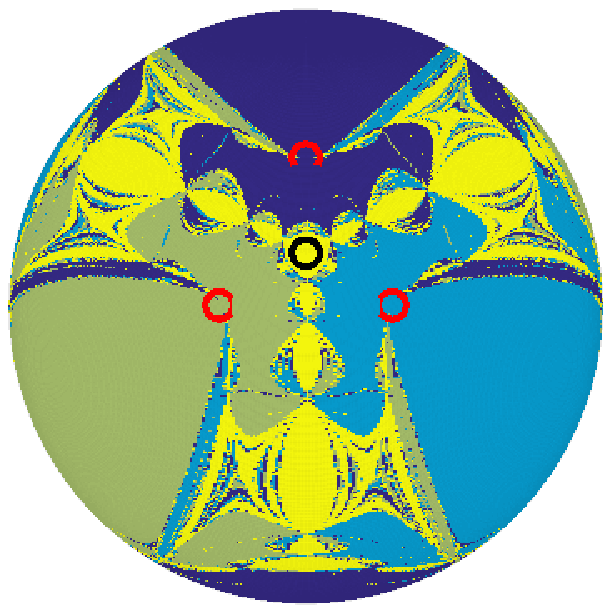}\hfill
	\includegraphics[width = 0.32\textwidth]{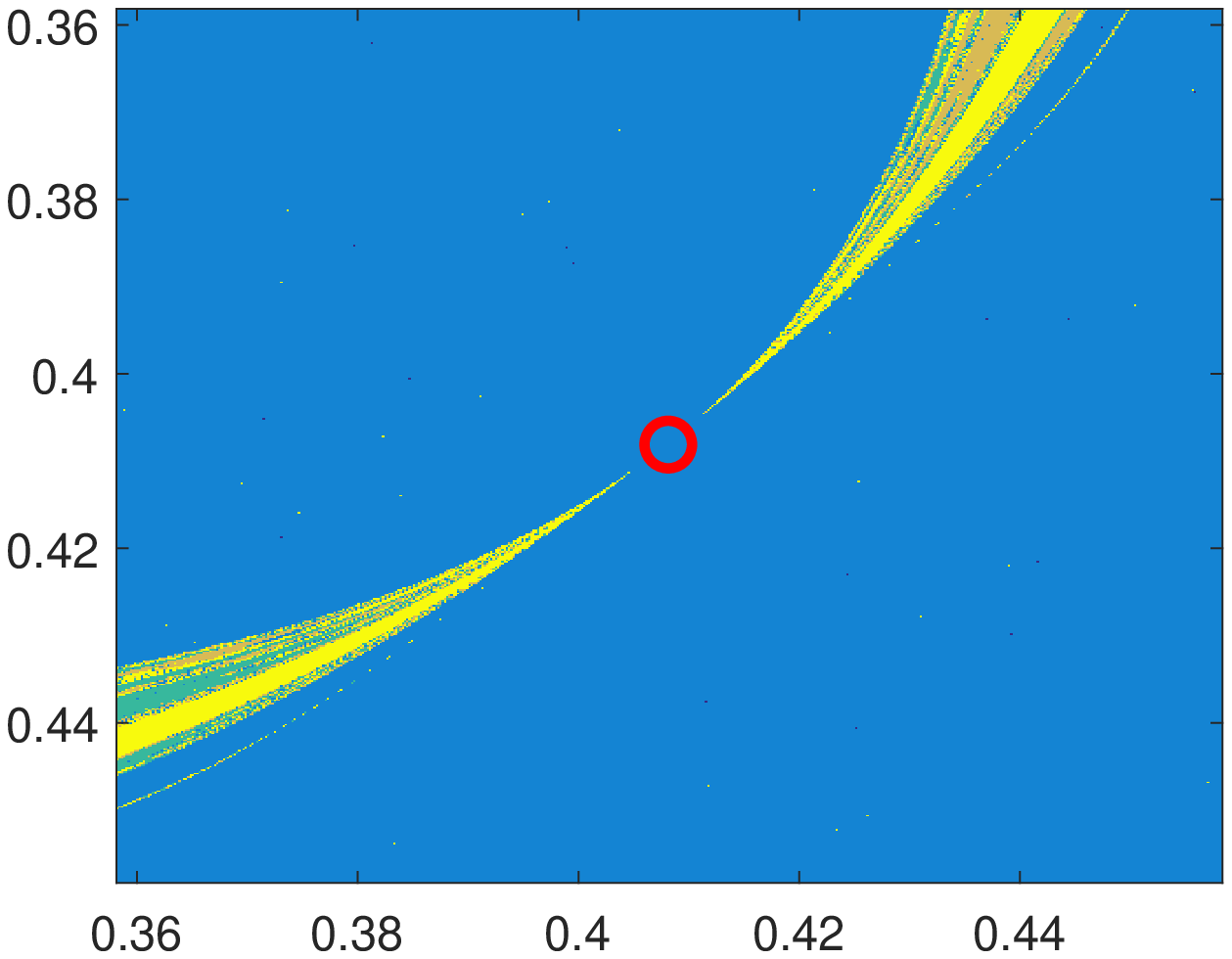}
	\caption{O--NCM's attraction region around an eigenvector who is Newton-stable (left), not Newton-stable (right), and the full unit sphere (middle) for $\T_{\weight=\weight_0}$ with dimension $n=3$. }	
	\label{fig:singular_point_attraction_b}    
\end{figure}

\section{Simulation results}\label{sec:results}
In this section we study numerically the performance of NCM and O--NCM for
computing the real eigenpairs of symmetric tensors, as compared to the homotopy method \cite{chen2016computing} and the adaptive shifted-HOPM \cite{kolda2014adaptive}.%
\footnote{Matlab code for the NCM methods can be found at  \url{https://github.com/arJaffe/NCM}, for the homotopy method at \url{http://users.math.msu.edu/users/chenlipi/TenEig.html}, and for the shifted HOPM  at  \url{http://www.sandia.gov/~tgkolda/TensorToolbox/index-2.6.html}.}
In the experiments we consider \textit{random Gaussian symmetric tensors}, whose entries are all i.i.d.\ $\mathcal N(0,1)$ up to the symmetry constraints. All experiments were done on a PC with an Intel i-3820 $3.6\text{GHz}$ processor, $16\text{GB}$ RAM and MATLAB version R2016a.




\paragraph{Finding all real eigenpairs}
We examine the time needed to compute \textit{all} eigenpairs for random tensors of order $m=4$ and various dimensions $n$. Similar results are obtained for other values of $m$.
For each tensor, we first ran the homotopy method to obtain all its eigenpairs. Next, we ran NCM and O--NCM, initialized repeatedly with random points on the unit sphere, until all eigenpairs were found.
Note that without running the homotopy method first, we would have no criterion to decide whether we actually found all tensor's eigenpairs.
The process was sequential, where a new run was initialized only after the previous one ended.
This process, however, can be easily parallelized. 
We stopped the NCM iterations when $\|\xv_{(k)}-\xv_{(k-1)}\|<\delta = 10^{-10}$ or if a maximal number of $k=k_{\max}=200$ iterations was reached.
In the latter case we declared that NCM failed to converge.
 For both NCM and O--NCM, and for all tensor dimensions we considered, only $\approx 0.2\%$ of all
 random initializations failed to converge within the maximal number of iterations.

Figure \ref{fig:4th_order_all_eigenvectors} (left) shows the number of real eigenpairs as averaged over 10 independent tensors for each value of $n$.
Figure \ref{fig:4th_order_all_eigenvectors} (right) shows on a 
{\em logarithmic} scale the average time 
it took to compute 
{all} real eigenpairs
via the homotopy, NCM and O--NCM,
for the same tensors.
These results
show that both NCM and O--NCM recovered all eigenpairs faster than the homotopy method by approximately two orders of magnitude.
Moreover, O--NCM did so much faster than NCM.
%
\begin{figure}
	\centering
	\includegraphics[width = 0.45\textwidth]{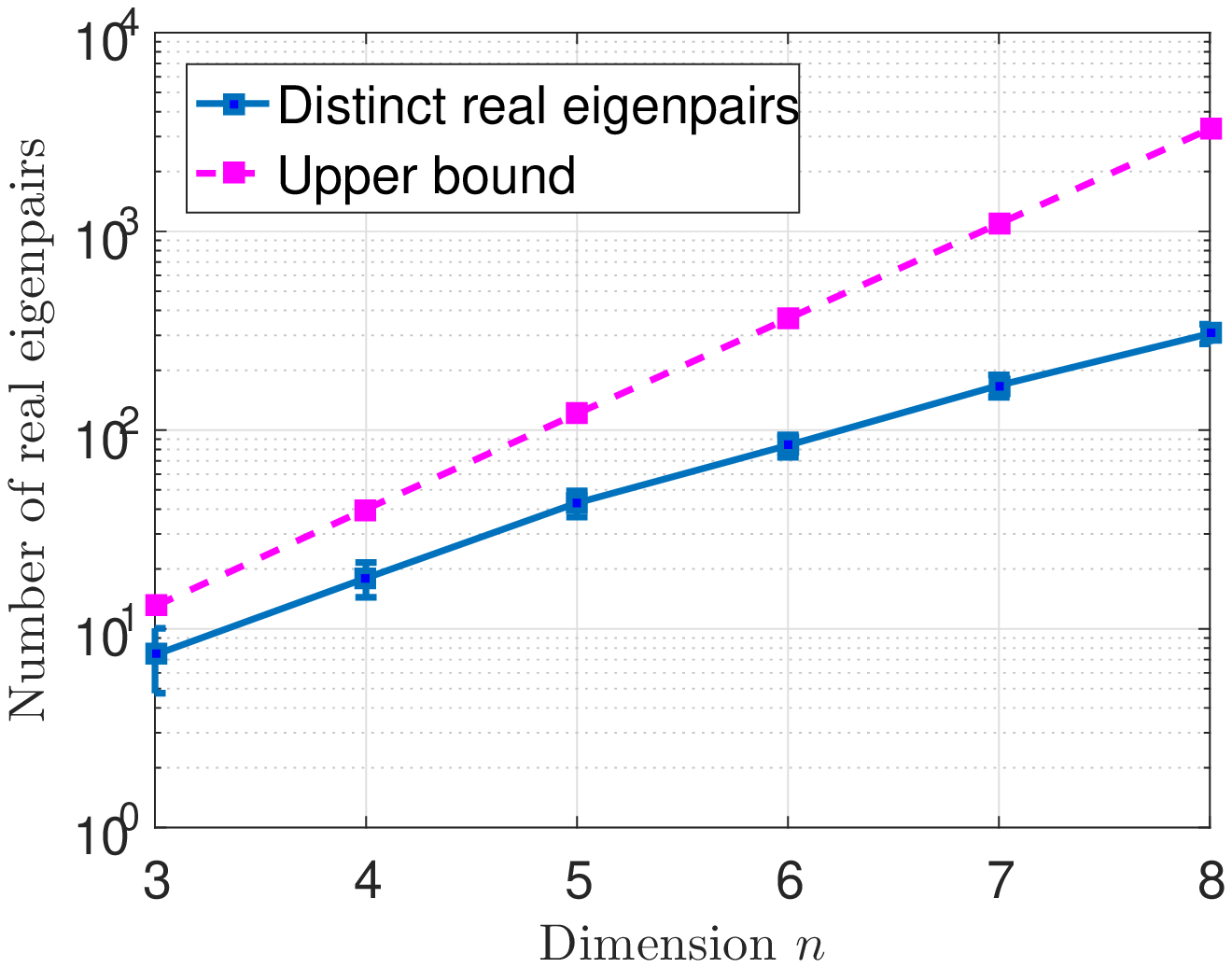}
	\hfill
	\includegraphics[width = 0.45\textwidth]{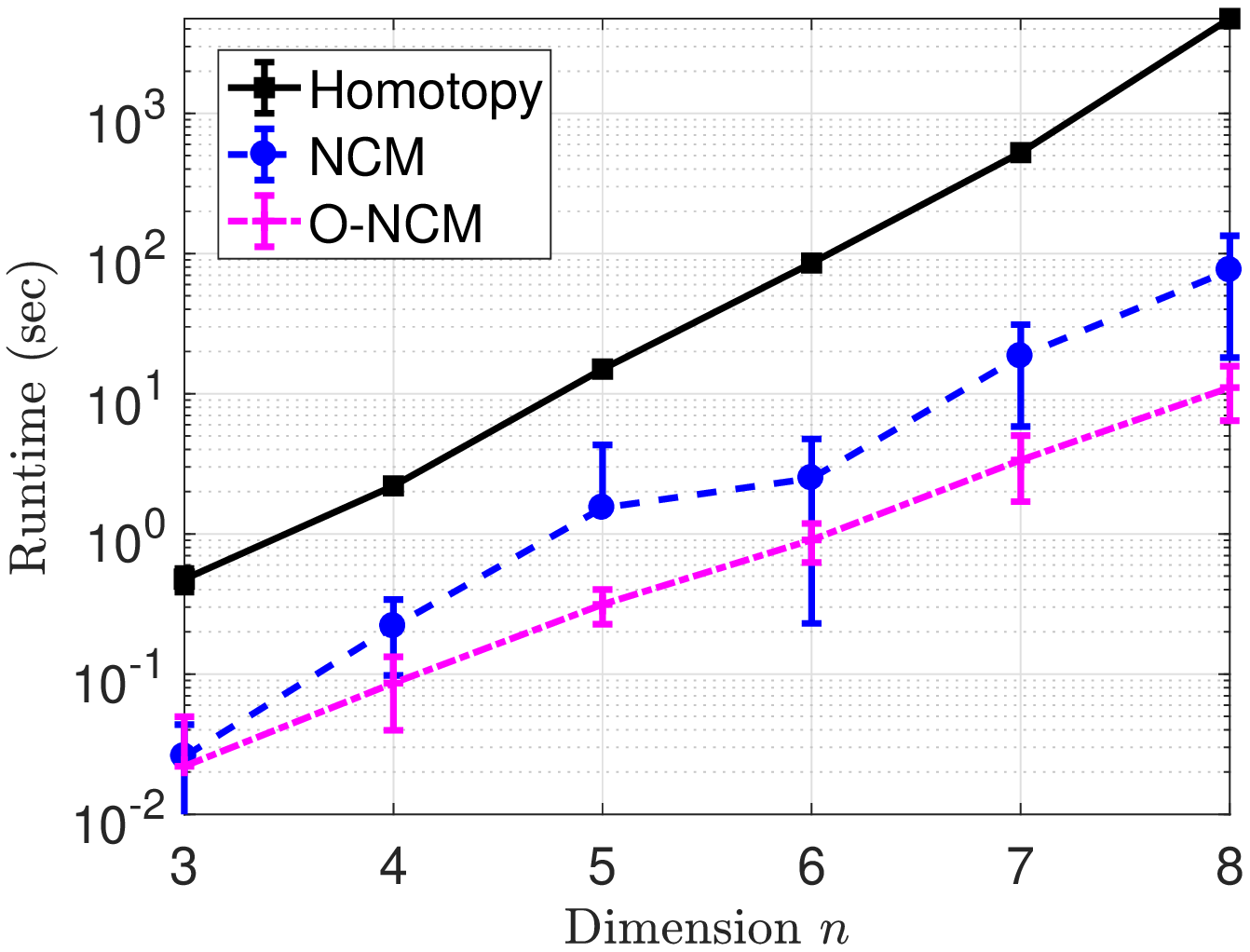}	
	\caption{ Left: The average number of real eigenpairs 
for random tensors of order $m=4$. Right: The average time to compute all eigenpairs via the homotopy, NCM and O--NCM. 
	}
	\label{fig:4th_order_all_eigenvectors}    
\end{figure}
\begin{figure}
	\centering
	\includegraphics[width = .45\textwidth]{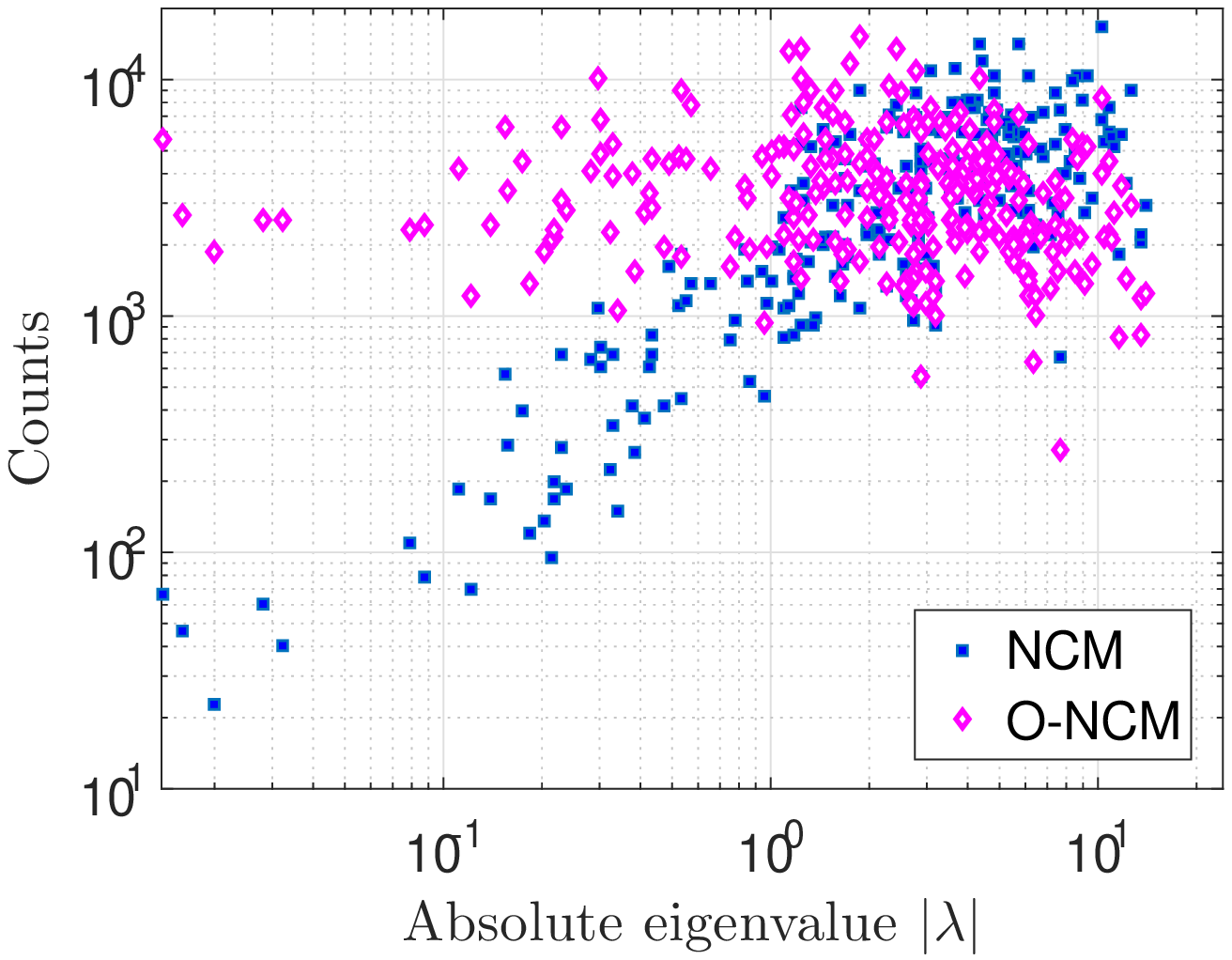}\hfill
	\includegraphics[width = .45\textwidth]{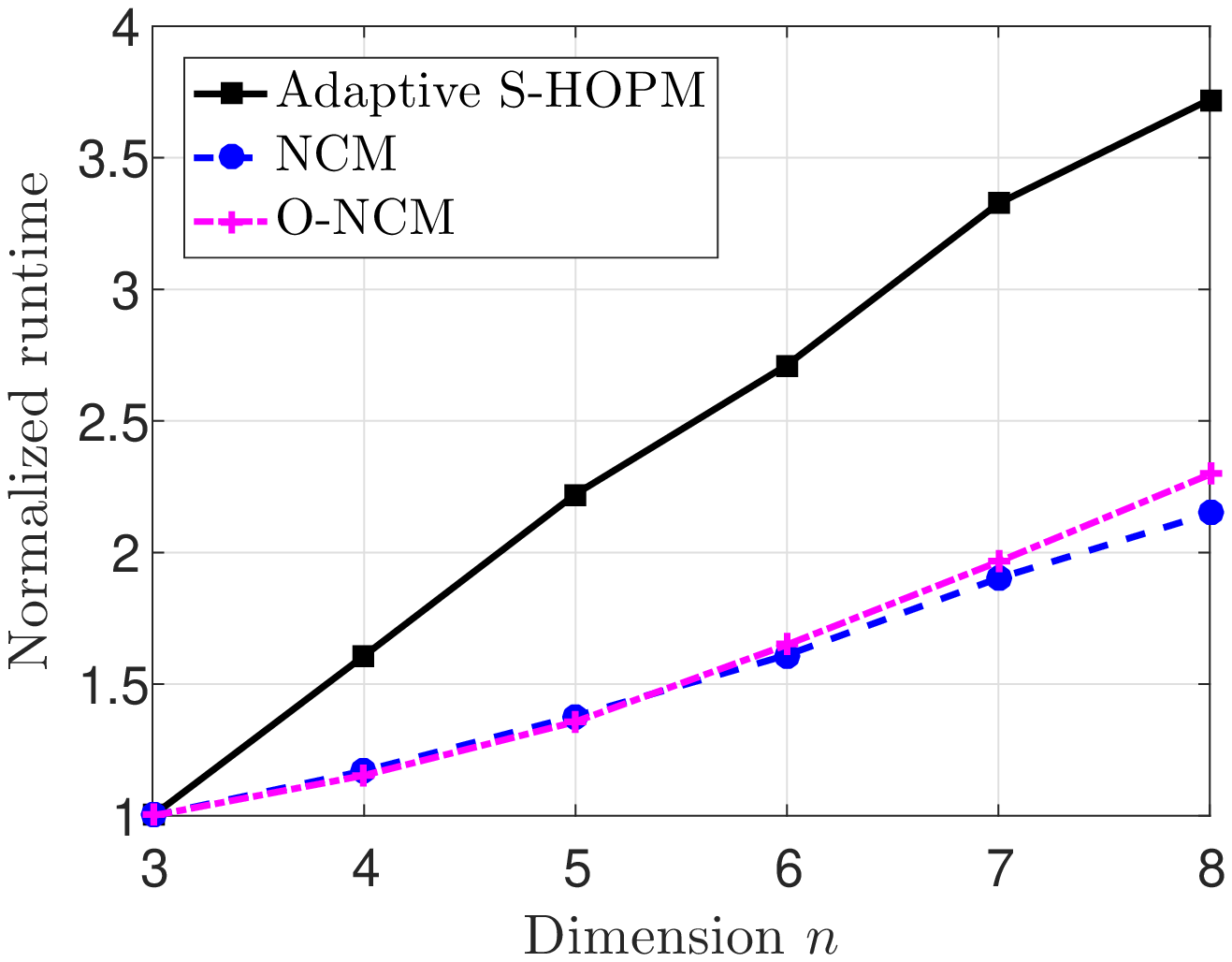}
	\caption{Left: Number of times NCM, O--NCM and adaptive S--HOPM converged to each of the eigenvalues of a random tensor with $m=4$ and $n=8$, over a total of $10^6$ random initial guesses. Right: Median runtime until convergence for NCM, O--NCM and the adaptive S--HOPM.}
\label{fig:compare_ncm_oncm}    
\end{figure}
\paragraph{Small eigenvalues}
To understand the gap in the runtime of NCM and O--NCM shown in Figure \ref{fig:4th_order_all_eigenvectors}, we next examine the dependence of both methods
on the eigenpair to which they converge.
As suggested by Theorems \ref{lem:simple_converge} and \ref{lem:orthogonal_converge}, we expect O--NCM to have larger attraction regions for small eigenvalues as compared to NCM.
Figure \ref{fig:compare_ncm_oncm} (left) shows on a log-log scale the relative number of times the two methods converged to each eigenvalue as a function of its absolute value
for a typical random tensor of order $m=4$ and dimension $n=8$.
These counts correspond to a total of $10^6$ random initializations uniformly distributed on the unit sphere. 
As one can see,
the probability for NCM to converge to an eigenpair decreases sharply when its eigenvalue becomes small,
while for O--NCM this probability seems to be independent of $|\lambda|$. 
This difference is the source of the gap in the runtime of the two methods for finding all eigenpairs.
For completeness, the eigenvalues found by the adaptive S-HOPM are also presented.

\paragraph{Convergence rates}
Figure  \ref{fig:compare_ncm_oncm} (right) shows the median runtime till convergence of the NCM and the shifted HOPM.
The stopping condition for all methods was set to $\|\xv_{(k)}-\xv_{(k+1)}\|<\delta = 10^{-10}$. The experiment was done on $100$ random tensors of fourth order with various dimensions. For each tensor, we initialized all methods with $100$ random starting points. To avoid the influence of any particular implementation, we normalized the results with the runtime of both methods for $n = 3$.
As illustrated in Fig. \ref{fig:compare_ncm_oncm},  the runtime increase of the NCM or O--NCM is significantly slower than the corresponding increase in the adaptive shifted-HOPM.
However, each NCM/O--NCM iteration may be slower, as it requires matrix inversion.

\section{Discussion and summary}\label{sec:discussion}
In this paper we developed and analyzed a Newton-based iterative approach to compute real eigenpairs of symmetric tensors. We now briefly discuss three important issues: its runtime, its ability to find all tensor eigenpairs, and its optimization point of view.


\paragraph{Runtime}
The computational complexity of each NCM or O--NCM iteration is 
dominated by two operations: computing the Hessian matrix in $\mathcal O(n^m)$ time
and solving a system of $n$ linear equations in $\mathcal O(n^3)$ time.
The latter step may be significantly sped up by applying various preconditioning techniques, as done in other iterative methods that solve systems of linear equations \cite{dembo1982inexact}.
For sparse tensors, the computation of the Hessian can be accelerated as well, see \cite{Smith2015splat}.

\paragraph{Optimization point of view} 
Following a constructive comment by one of the referees, we note that NCM can be seen as an adaptation of the Gauss--Newton method \cite{bjorck1996numerical}.
Recall that $\grad(\xv) = \T(I,\xv,\dots,\xv) - \mu(\xv) \xv=\bm 0$ if and only if $\xast\in S_{n-1}$ is an eigenvector of $\T$, with a corresponding eigenvalue $\lambda^* = \mu(\xast)$. Hence, our goal is to find the \emph{global} minima of the realizable nonlinear least-squares problem
\beqn
\min_{\xv\in S_{n-1}} \tfrac{1}{2}\nrm{\grad(\xv)}^2.
\label{eq:obj_fun}
\eeqn
Given the current estimate $\xk\in S_{n-1}$, Gauss--Newton first linearizes $\grad(\xv)$ at $\xk$,
\beq
\grad(\xv)\approx \grad(\xk) + A(\xk)(\xv-\xk) = \grad(\xk) + A(\xk)\yv,
\eeq
where $[A(\xv)]_{ij} = \partial g_{i}(\xv)/\partial x_j$ is the $n\times n$ Jacobian matrix of $\grad(\xv)$, given in \eqref{eq:A_def}.
Then, instead of \eqref{eq:obj_fun}, the following approximate linear least-squares problem is solved,
\beq
\yk = \argmin_{\yv\in \R^n} \tfrac{1}{2}\nrm{\grad(\xk) + A(\xk)\yv}^2,
\eeq
which is exactly the NCM correction in \eqref{eq:system_lin_simple}.

Besides NCM, other nonlinear optimization methods can be used to solve \eqref{eq:obj_fun}, such as the Levenberg-Marquardt algorithm
\cite{marquardt1963algorithm} and 
other trust-region and line search algorithms.
These methods, among other things, introduce an additional regularization term to better control the direction in which the method proceeds at each iteration,
similarly to the role played by the (adaptive) shifted-HOPM as compared to HOPM.
Specifically, instead of \eqref{eq:system_lin_simple}, one solves the following linear system with an appropriate regularization matrix $B_k\in\R^{n\times n}$, 
\beq
( A(\xk)^T A(\xk) + B_k)\yv = - A(\xk)^T \grad(\xk).
\eeq
While NCM currently has no global convergence guarantees, an appropriate (adaptive) choice of $B_k$ can lead to global convergence guarantees, including to eigenvectors having a zero Hessian. Further studying the role of regularization for the tensor eigen-problem is an interesting direction for future research.

\begin{figure}%
\centering
\includegraphics[width=.48\columnwidth]{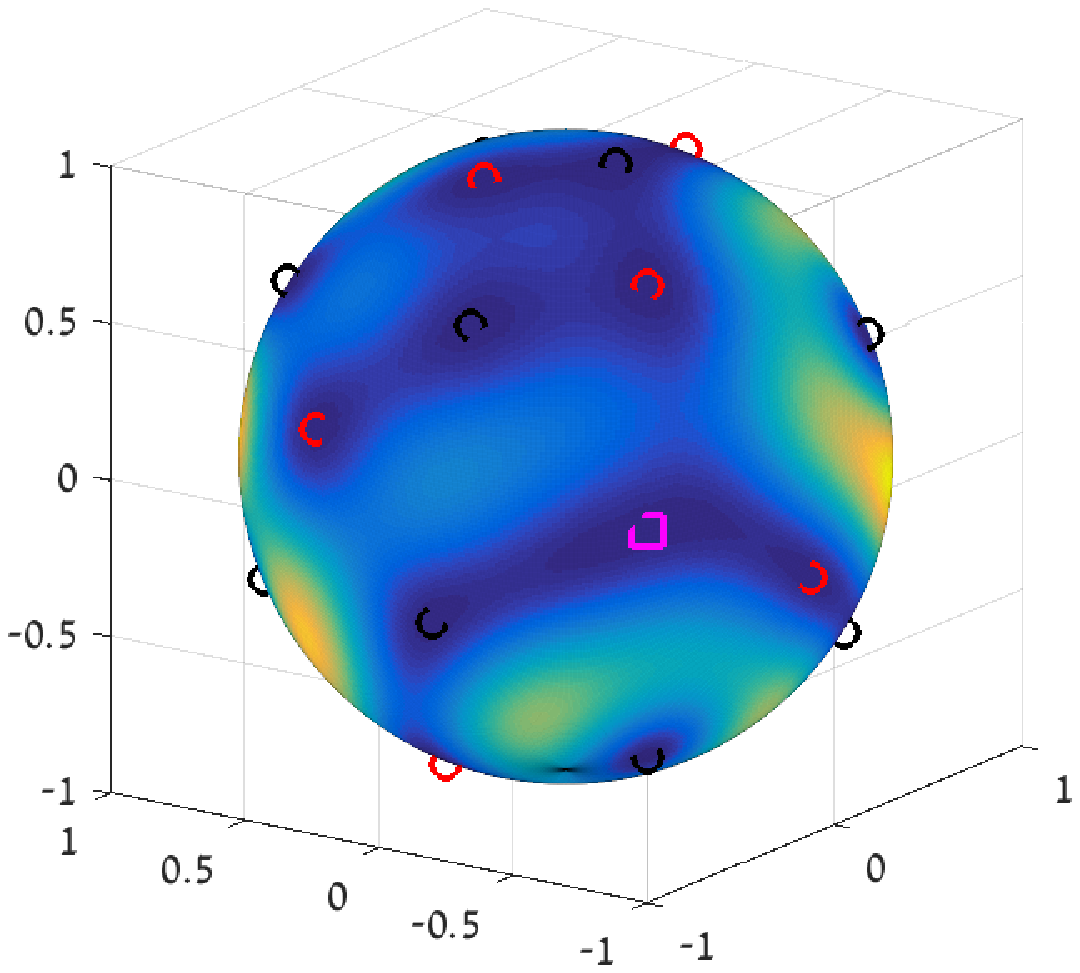}%
\includegraphics[width=.48\columnwidth]{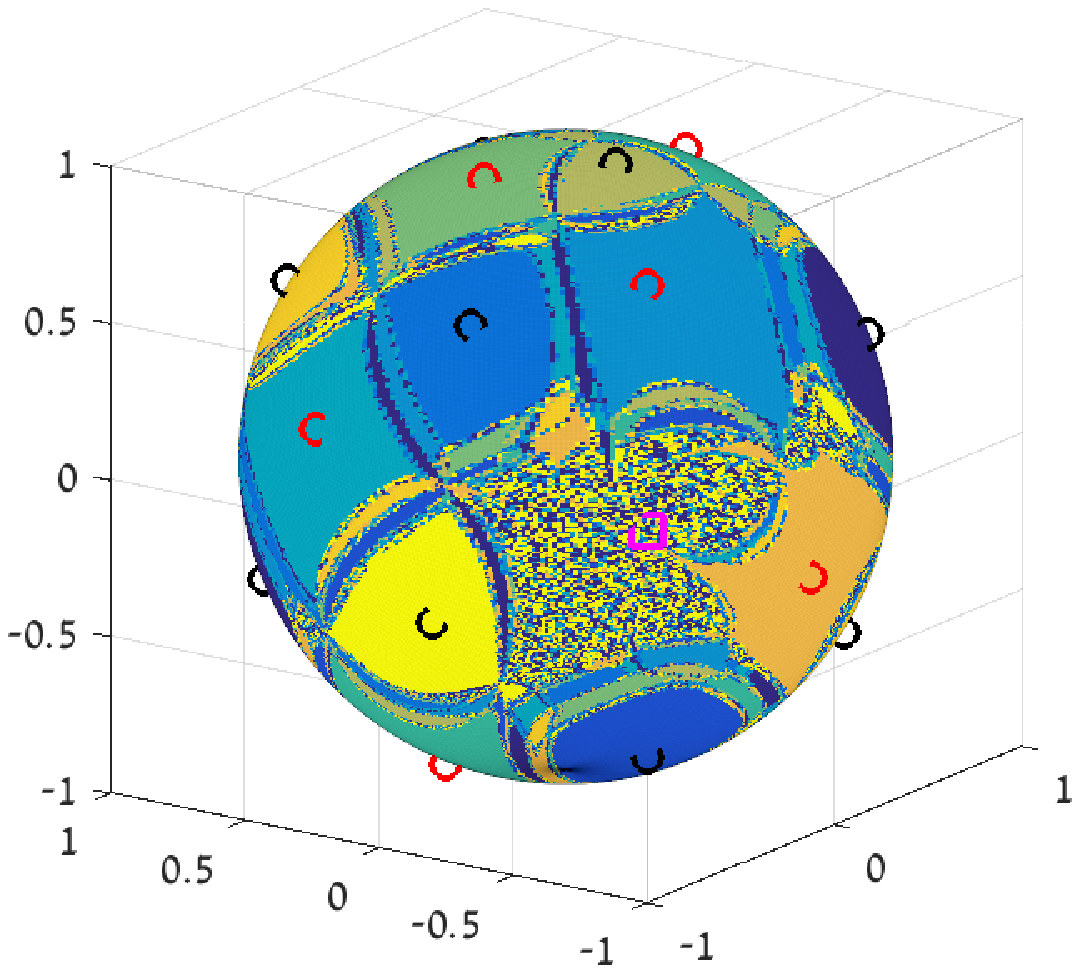}%
\caption{Left: Plot of $f(\xv)$ for \cite[Example 1]{kofidis2002best}. The global minima (circles) correspond to eigenpairs. A local minimum is depicted by a square. Right: Attraction regions for NCM.}%
\label{fig:kofidis}%
\end{figure}

However, in addition to the global minima, $f(\xv)\equiv\frac{1}{2}\nrm{\grad(\xv)}^2$ may have local minima which should be avoided.
Interestingly, NCM elegantly avoids such local minima as the following example illustrates.
In Figure \ref{fig:kofidis} (left), we plot $f(\xv)$ as a function of $\xv\in S_2$ for the $3\times 3\times 3\times 3$ symmetric tensor of Example 1 in \cite{kofidis2002best}.
Its eigenvectors are depicted by circles while a local minimum $\xlocal$ of $f$ is depicted by a square symbol. 
In Figure \ref{fig:kofidis} (right) we show the attraction regions for NCM starting from various locations on $S_{n-1}$.
As one can see, NCM does not converge to $\xlocal$ and in fact is highly unstable around this point;
close initial points in this neighborhood may converge to arbitrarily far eigenvectors.


To see why this is so, note that since $\xlocal$ is a local minimum, for an initial point $\xz$ near $\xlocal$, NCM may get closer and closer to $\xlocal$ at the first few iterations. However, the facts that $\grad(\xlocal)$ is bounded away from $\bm 0$ 
and $\nabla f(\xlocal) = A(\xlocal)^T \grad(\xlocal) = \bm 0$ implies that $\grad(\xlocal)$ is in the null space of $A(\xlocal)^T$.
As $\xk$ gets closer to $\xlocal$, $A(\xk)$ becomes close to singular. The result is an \emph{overshoot}, a sharp increase in $\nrm{\yk} = \nrm{A(\xk)^{-1} \grad(\xk)}$, taking $\xkpo$ far away from $\xk$ and $\xlocal$.  

\paragraph{Finding all eigenpairs of generic tensors}\label{sec:generic}
According to our theoretical analysis, NCM and O--NCM converge to eigenpairs whose Hessian matrix is full rank. An interesting question is whether these methods can thus converge to \emph{all} real eigenpairs of a \emph{generic} symmetric tensor
\cite{cartwright2013number}.
%
%
Interpreting generic in the sense of algebraic geometry,
an adaptation of \cite[Theorem 1.2]	{cartwright2013number} to the symmetric tensor case implies the following (proof omitted).

\begin{proposition}
\label{prop:generic}
All real eigenpairs of a generic symmetric tensor are Newton-stable.
\end{proposition}

Hence, Theorems \ref{lem:simple_converge} and \ref{lem:orthogonal_converge} imply that NCM and O--NCM  are guaranteed to find {all} eigenpairs of a generic symmetric tensor given a sufficiently large number of random initializations.


\paragraph{Acknowledgments} We thank Lek--Heng Lim, Meirav Galun and Haim Avron for interesting discussions.

\appendix
\section{\normalfont Convergence of NCM}
\label{sec:NCM_proofs}
To prove Theorem \ref{lem:simple_converge} we shall make use of the following auxiliary lemma. 
\begin{lemma} \label{lem:NCM_aux}
Consider one update step of Algorithm \ref{algo:newton_correction},
as in Equation (\ref{eq:normalized_update}), 
starting from an initial $\xv\in S_{n-1}$ and ending with $\xv'=(\xv + \yv)/\nrm{\xv + \yv}\in S_{n-1}$. Let $\yast=\xast-\xv$. If \(\nrm{\yv-\yast}\le 1/2\), then
\beqn
\label{eq:x_tag_xast}
\nrm{\xast-\xv'} 
\leq \frac{2\nrm{\yast - \yv}}{1 - {\nrm{\yast - \yv}}}.
\eeqn
\end{lemma}
\begin{proof}
By definition, 
\begin{eqnarray}
\label{eq:bound2}
\nrm{\xast-\xv'} 
& = & \left\|\xast-\frac{\xv + \yv}{\nrm{\xv + \yv}}\right\|
= \left\|\xast-\frac{\xast -\yast + \yv}{\nrm{\xast -\yast + \yv}}\right\| \nonumber \\
& = & \left\|{\xast\left(1-\frac{1}{\nrm{\xast-\yast+\yv}}\right) 
        + \frac{\yast-\yv}{\nrm{\xast-\yast+\yv}}} \right\|.
\end{eqnarray}
Since $\nrm{\xast} = 1$, by the triangle inequality,
\beq
1 - \nrm{\yast - \yv} \leq \nrm{\xast -\yast + \yv} \leq 1 + \nrm{\yast - \yv}.
\eeq
Applying the triangle inequality to (\ref{eq:bound2}), combined with the assumption  $\nrm{\yast-\yv}\le 1/2$,
\beq
\nrm{\xast-\xv'}
&\leq& \left| 1 -\frac{1}{\nrm{\xast -\yast + \yv}}\right| + \frac{\nrm{\yast - \yv}}{\nrm{\xast -\yast + \yv}}
\\
&\leq& \left( \frac{1}{1 - \nrm{\yast - \yv}}-1\right) + \frac{\nrm{\yast - \yv}}{1 - \nrm{\yast - \yv}}
=
\frac{2\nrm{\yast - \yv}}{1 - {\nrm{\yast - \yv}}},
\eeq
hence concluding the proof.
\end{proof}

\begin{proof}[Proof of Theorem \ref{lem:simple_converge}]
To prove quadratic convergence it suffices to show that there exists an $\eps >0$ and a constant \(C>0\) such that from any initial point $\xz$ that satisfies $\nrm{\xz-\xast} < \varepsilon$,
\beq
\err_k = \frac{\nrm{\xast-\xkpo}}{\nrm{\xast-\xk}^2}<C, \quad \forall k\ge 0.
\eeq

We start by analyzing $\err_k$ at the first iteration $k=0$. 
Let $\yz$ be the approximate correction of $\yast=\xast-\xz$, given by the solution of \eqref{eq:system_lin_simple}. The new approximation of $\xast$, given by Eq. (\ref{eq:normalized_update}), is $\xo=(\xz + \yz)/\nrm{\xz + \yz}$.
Assume for the moment that the initial guess $\xz$ is sufficiently close to $\xast$ so \ that $\nrm{\yast-\yz}<1/2$. Then, by Lemma \ref{lem:NCM_aux},
\beqn
\label{eq:xast_xo}
\nrm{\xast-\xo} 
\leq \frac{2\nrm{\yast - \yz}}{1 - {\nrm{\yast - \yz}}}.
\eeqn
%
%
Hence, it suffices to bound $\nrm{\yast - \yz}$.
To this end, we view the exact system of non-linear equations (\ref{eq:system_quad_simple}), whose solution is $\yast$, as a perturbation of the approximate system of linear equations (\ref{eq:system_lin_simple}), whose solution is $\yz$.
Consider the matrix $A$ of Eq.\!\! \eqref{eq:A_def} evaluated at the eigenvector $\xast$,
\beq
A(\xast) = H(\xast)- m \lamast \xast (\xast)^T.
\eeq
Note that $A(\xast)$ is symmetric with eigenvalues $(\mu_1^\ast,\ldots,\mu_{n-1}^\ast, -2\lamast)$.
Since $\xast$ is $\g$-Newton-stable, $|{\mu_{i}^\ast}| \geq \g$ for all $i\in[n-1]$.
In addition, since $\lamast \neq 0$ by assumption, $A(\xast)$ is full rank with smallest singular value 
\beqn
\label{eq:ncm_smallest_ev}
\sigma^*=\sigma_{\min}(A(\xast)) \geq \min\{\g, 2|\lamast|\}>0.
\eeqn
By the continuity of $\sigma_{\min}(A(\xv))$ in $\xv$, there exists a $\eps_1 > 0$ such that $\sigma_{\min}(A(\xv)) \geq \sigma^*/2$ for all $\xv$ with $\nrm{\xast- \xv} \leq \eps_1$. In particular, if $\nrm{\xast-\xz}<\eps_1$, then $A(\xz)$ is invertible and the solution to \eqref{eq:system_quad_simple} satisfies the following implicit equation in $\yast$, 
\beq
 \yast &=& - A(\xz)^{-1} (\grad(\xz) - \Delta(\xz,\yast)).
\eeq
Similarly, the unique solution to the correction equation \eqref{eq:system_lin_simple} is as in \eqref{eq:system_lin_simple_solution},
\beq
\yz = -A(\xz)^{-1} \grad(\xz).
\eeq
Subtracting the last two equations gives 
\beqn
\label{eq:yast_yz}
\nrm{\yast - \yz} \leq \nrm{A(\xz)^{-1}}\cdot \nrm{\Delta(\xz,\yast)}
\leq \frac{2}{\sigma^*} \nrm{\Delta(\xz,\yast) }.
\eeqn

To bound $\nrm{\Delta(\xz,\yast)}$,
first note 
that for any symmetric tensor $\T$ there exists an $M = M(\T) < \infty$ such that for any $\xv\in S_{n-1}$, $\yv\in\R^n$ and $j \leq m-1$,
\beqn
\label{eq:tensor_spectral}
\big\|\T(I,\underbrace{\xv,\ldots,\xv}_{m-j-1},\underbrace{\yv,\dots,\yv}_{j \text{ times}})\big\|^{} \leq M \nrm{\yv}^j.
\eeqn
Similar bounds hold for $\T(\xv,\ldots,\xv,\yv,\dots,\yv)\xv$ and 
$\T(\xv,\ldots,\xv,\yv,\dots,\yv)\yv$ according to their powers in $\yv$.
Bounding each term of $\Delta(\xz,\yast)$ in \eqref{eq:simple_high_order_terms} separately by \eqref{eq:tensor_spectral},
there are less than $3m^2$ terms involving $M\nrm{\yast}^2$ and
at most $3\cdot 2^m$ terms  involving $M\nrm{\yast}^j$ with $j\in\{3,\dots,m\}$.
Assuming $ \nrm{\yast}=\nrm{\xast-\xz} <m^{2}/2^m \le 1$, implies
\beq
\nrm{\Delta(\xz,\yast)} \leq 3 m^2 M \nrm{\yast}^2 +  3\cdot 2^m M\nrm{\yast}^3 \leq   6 m^2 M \nrm{\yast}^2.
\eeq
Inserting this bound into (\ref{eq:yast_yz}), 
\beqn
\label{eq:yast_yz_2}
\nrm{\yast - \yz} \leq \frac{12 m^2 M}{\sigma^*}\nrm{\yast}^2. 
\eeqn
Note that if $\nrm{\xast-\xz}=\nrm{\yast}\le \eps_2=(\sigma^*/24m^2M)^{1/2}$, then   $\|\yast - \yz\|\leq 1/2$ as required by Lemma \ref{lem:NCM_aux}.
Under this condition, by Eq. (\ref{eq:xast_xo}), it follows that
\beqn
\label{eq:err_th_simple}
\err_0 = \frac{\nrm{\xast-\xo}}{\nrm{\xast-\xz}^2}
\leq \frac{1}{\nrm{\yast}^2}
\frac{2\nrm{\yast - \yz}}{1 - {\nrm{\yast - \yz}}}
\leq \frac{48m^2 M }{\sigma^*} 
.
\eeqn
As an interim summary, if $\nrm{\xast-\xz}\le \min\{\eps_1,\eps_2,m^2/2^m\}=\eps_0$, then Eq.\ (\ref{eq:err_th_simple}) holds. We conclude the proof for a general iteration $k\geq 1$ by induction.
For the first induction step to work, it required that if $\nrm{\xast-\xz}\leq \eps <\eps_0$, then $\nrm{\xast-\xo}<\eps$ as well.
By \eqref{eq:err_th_simple}, this is satisfied for $\eps=\min\{\eps_0,\sigma^*/48m^2M\}$ and the proof for a general $k$ holds similarly. The quadratic convergence of Algorithm \ref{algo:newton_correction} follows.
\end{proof}

\section{\normalfont Proof of Lemma \ref{lem:stewart}}
\label{sec:ONCM_proofs}
We show that a vector $\ortc$ satisfies \eqref{eq:system_lin} if and only if it satisfies
        \begin{equation}\label{eq:alt_lin_system}
        P_{\xv}^\perp \tA(\xv)P_{\xv}^\perp \ortc = - P_{\xv}^\perp \grad(\xv)
        \quad \text{and} \quad \xv^T \ortc = 0.
        \end{equation}
Lemma \ref{lem:stewart} then follows by recalling that $P_{\xv}^\perp = U_{\xv} U_{\xv}^T$ and multiplying the first equation in        \eqref{eq:alt_lin_system} by $U_{\xv}^T$ from the left.

To prove the first direction, note that by the last row of \eqref{eq:system_lin}, the solution $\ortc$ to \eqref{eq:system_lin} is perpendicular to $\xv$, so $\xv^T \ortc = 0$ and $P_{\xv}^\perp \ortc = \ortc$.
Multiplying the first ``row'' of \eqref{eq:system_lin} by $P_{\xv}^\perp$ from the left and noting that $P_{\xv}^\perp \xv = \bm 0$, we find that the left hand side is given by
\begin{equation}
P_{\xv}^\perp( \tA(\xv) \ortc  - \beta\xv)  = P_{\xv}^\perp \tA(\xv)  \ortc = P_{\xv}^\perp \tA(\xv) P_{\xv}^\perp \ortc.
\end{equation}
In addition, one can easily check that $\grad(\xv)$ is perpendicular to $\xv$, so the right hand side of the equality in \eqref{eq:alt_lin_system} is $-P_{\xv}^\perp \grad(\xv) = -\grad(\xv)$
and \eqref{eq:alt_lin_system} follows.

To prove the other direction, 
suppose $\ortc$ satisfies \eqref{eq:alt_lin_system}.
So $\ortc^T \xv = 0$ and $P_{\xv}^\perp \ortc = \ortc$.
Define $\beta = \xv^T \tA(\xv)\ortc $ and write the left hand side of \eqref{eq:alt_lin_system} as
\beq
P_{\xv}^\perp \tA(\xv) P_{\xv}^\perp \ortc = (I-\xv \xv^T) \tA(\xv) P_{\xv}^\perp \ortc = \tA(\xv) \ortc - \xv \xv^T \tA(\xv) \ortc =  \tA(\xv) \ortc - \beta \xv.
\eeq
Since $-P_{\xv}^\perp \grad(\xv) = -\grad(\xv)$, it follows that  $(\ortc,\beta)$ satisfies \eqref{eq:system_lin} as required.

\section{\normalfont Convergence of O--NCM}
\label{sec:ONCM_proofs_conv}
The proof of Theorem \ref{lem:orthogonal_converge} is similar to that of Theorem \ref{lem:simple_converge},
and makes use of the following auxiliary lemma.
\begin{lemma}
\label{lem:ONCM_aux}
Consider one update step of Algorithm \ref{algo:newton_orthogonal},
as in Equation (\ref{eq:ortho_update}), 
starting from an initial $\xv\in S_{n-1}$ and ending with $\xv'=(\xv + \ortc)/\nrm{\xv + \ortc}\in S_{n-1}$.
Let $\alpha=\xv^T \xast$ and $\ortc^* = \alpha \xast - \xv$.
If $\alpha \geq 1/2$ and \(\nrm{\ortc^*-\ortc} \leq 1/4\), then
\beqn
\label{eq:oncm-x_tag_xast}
\nrm{\xast-\xv'} 
\leq \frac{2\nrm{\ortc^* - \ortc}}{\alpha - {\nrm{\ortc^* - \ortc}}}
\leq 8\nrm{\ortc^* - \ortc}.
\eeqn
\end{lemma}

\begin{proof}
By definition,
\beqn
\nonumber
\|\xast-\xv'\| 
&=& \left\|\xast-\frac{\xv + \ortc}{\|\xv + \ortc\|}\right\|
= \left\|\xast-\frac{\alpha\xast -\ortc ^* + \ortc}{\|\alpha\xast -\ortc ^* + \ortc\|}\right\|
\\
\label{eq:aux_1}
&=& 
\left\|\xast\left(1-\frac{\alpha}{\|\alpha\xast -\ortc ^* + \ortc\|}\right)
+ \frac{\ortc ^* - \ortc}{\|\alpha\xast -\ortc ^* + \ortc\|}
\right\|
.
\eeqn
Since $\nrm{\xast} = 1$ and $\alpha>0$, by the triangle inequality,
\beq
\alpha - \|\ortc ^* - \ortc\| \leq \|\alpha\xast -\ortc ^* + \ortc\| \leq \alpha + \|\ortc ^* - \ortc\|.
\eeq
Applying the triangle inequality to \eqref{eq:aux_1}, combined with the assumption \(\nrm{\ortc^*-\ortc} \leq \alpha/2\),
\beq
\|\xast-\xv'\| 
&\leq& \left| 1 -\frac{\alpha}{\|\alpha\xast -\ortc ^* + \ortc\|}\right| + \frac{\|\ortc^* - \ortc\|}{\|\alpha\xast -\ortc ^* + \ortc\|}.
\\
&\leq& \left( \frac{\alpha}{\alpha - \|\ortc ^* - \ortc\|} - 1\right) + \frac{\|\ortc^* - \ortc\|}{\alpha - \|\ortc ^* - \ortc\|}
=
\frac{2\|\ortc ^* - \ortc\|}{\alpha - {\|\ortc ^* - \ortc\|}},
\eeq
hence concluding the proof. 
\end{proof}

\begin{proof}[Proof of Theorem \ref{lem:orthogonal_converge}]
We show that 
there exists an $\varepsilon  >0$ and a constant $C>0$, such that for any initial point $\xz$ that satisfies $\| \xz-\xast\|<\varepsilon$,
\begin{equation}
\err_k = \frac{\|\xast-\xkpo\|}{\|\xast-\xk\|^2} < C
,\quad \forall k\geq 0.
\end{equation}

\newcommand{\ortcz}{\ortc_{(0)}}
\newcommand{\zz}{\bm z_{(0)}}
We start by analyzing $\err_k$ at the first iteration $k=0$. 
Let $\ortcz = U_{\xz} \zz$ be the approximate correction of $\ortc^* = \alpha \xast - \xz$, given by the solution of \eqref{eq:Cz_d}.
The new approximation of $\xast$
is $\xo = (\xz +\ortcz)/\nrm{\xz+\ortcz}$.
Since $\ortc^*$ is orthogonal to $\xast$,
the denominator of $\err_0$ satisfies
\beqn
\label{eq:lb_1}
\|\xast - \xz\|^2 = \|\xast - (\alpha \xast - \ortc^*)\|^2
= (1-\alpha)^2 + \|\ortc^*\|^2
\geq \|\ortc^*\|^2.
\eeqn

To bound the numerator of $\err_0$,
assume for the moment that $\xz$ is sufficiently close to $\xast$ so that $\alpha=\xz^T \xast\geq 1/2$ and  \(\nrm{\ortc^*-\ortcz} \leq \alpha/2\). Then, by Lemma \ref{lem:ONCM_aux},
\beqn
\label{eq:oncm_xast_xo}
\nrm{\xast-\xo} 
\leq \frac{2\nrm{\ortc^* - \ortcz}}{\alpha - {\nrm{\ortc^* - \ortcz}}}
\leq 8\nrm{\ortc^* - \ortcz}
.
\eeqn
Hence, it suffices to bound $\nrm{\ortc^* - \ortcz}$.
Define $\bm z^* = U_{\xz}^T \ortc^*$ and note that since $\xast$ and $\ortc^*$ are orthogonal, $\xz^T \ortc^* = (\alpha\xast - \ortc^*)^T \ortc^* = -\|\ortc^*\|^2$.
Writing $I = U_{\xz} U_{\xz}^T  + \xz \xz^T$, we thus have
\beqn
\label{eq:bxast_bzast}
\ortc^* = (U_{\xz} U_{\xz}^T  + \xz \xz^T) \ortc^* = U_{\xz} \bm z^* - \|\ortc^*\|^2 \xz.
\eeqn
Since $\nrm{\xz}=1$, 
\beqn
\label{eq:bx_bxast}
\|\ortc^*-\ortcz\| = \Big\|U_{\xz}\bz^* - \|\ortc^*\|^2 \xz - U_{\xz}\zz\Big\| 
\leq \|\bz^* - \zz\| + \|\ortc^*\|^2.
\eeqn
To bound $\|\bz^*-\zz\|$, we view the exact system of non-linear equations (\ref{eq:system_quad}), whose solution is $\ortc^*$, as a perturbation of the approximate system of linear equations \eqref{eq:Cz_d}, whose solution is $\ortcz = U_{\xz} \zz$.
By \eqref{eq:system_quad}, $\ortc^*$ solves the non-linear equation
\beqn
\label{eq:bxast_system_lin}
\tA(\xz) \ortc^* = -\grad(\xz) + \beta^* \xz 
+ \tD(\xz,\ortc^*, \beta^*).
\eeqn
We multiply \eqref{eq:bxast_system_lin} by $U_{\xz}^T$ from the left and plugin \eqref{eq:bxast_bzast} to obtain the set of non-linear equations in $\bm z^*$ (and $\beta^*,\ortc^*$),
\beqn
\label{eq:Hpxz}
H_p(\xz)\bz^* =-U_{\xz}^T \Big(\grad(\xz) - \tD(\xz,\ortc^*, \beta^*)\Big).
\eeqn
Since $\xast$ is $\g$-Newton-stable, 
the projected Hessian $H_p(\xast)$ is full rank with smallest singular value \[\sigma_{\min}(H_p(\xast)) = \g>0.\]
By the continuity of $\sigma_{\min}(H_p(\xv))$ in $\xv$, there exists an $\eps_1 > 0$ such that $\sigma_{\min}(H_p(\xz)) \geq \g/2$ for all $\xv$ with $\nrm{\xast- \xv} \leq \eps_1$. In particular, if $\nrm{\xast-\xz}<\eps_1$, then $H_p(\xz)$ is invertible and the solution to \eqref{eq:Hpxz} satisfies the following implicit equation in $\bz^*$, 
\beq
\bz^* = - H_p(\xz)^{-1} U_{\xz}^T \Big(\grad(\xz) - \tD(\xz,\ortc^*, \beta^*)\Big).
\eeq
Similarly, the unique solution to \eqref{eq:Cz_d} is  
\beq
\zz = - H_p(\xz)^{-1} U_{\xz}^T \grad(\xz).
\eeq
Subtracting the last two equations gives 
\beqn
\label{eq:zast_zz}
\nrm{\bz^* - \zz} \leq \nrm{H_p(\xz)^{-1}}\cdot \nrm{\tD(\xz,\ortc^*, \beta^*)}
\leq \frac{2}{\g} \nrm{\tD(\xz,\ortc^*, \beta^*)}.
\eeqn
We bound the norm of $\tD(\xz,\ortc^*, \beta^*)$ in \eqref{eq:Delta_ONCM} by 
\beqn
\label{eq:Delta_norm_oncm}
\nrm{\tD(\xv, \ortc^*, \beta^*)} \leq  |\beta^*|\cdot \nrm{\ortc^*}  + \sum_{i=2}^{m-1} \tbinom{m-1}{i}\Big\|\T(I,\underbrace{\xv,\ldots,\xv}_{m-i-1},\underbrace{\ortc^*,\ldots,\ortc^*}_{i})\Big\|.
\eeqn
To bound the terms in the sum, 
note 
that there exists an $M = M(\T) < \infty$ such that for any $\xv\in S_{n-1}$, $\bm u\in\R^n$ and $j \leq m-1$,
\beqn
\label{eq:oncm_tensor_spectral}
\big\|\T(I,\underbrace{\xv,\ldots,\xv}_{m-j-1},\underbrace{\bm u,\dots,\bm u}_{j \text{ times}})\big\|^{} \leq M \nrm{\bm u}^j.
\eeqn
Bounding each term in the sum in \eqref{eq:Delta_norm_oncm} by \eqref{eq:oncm_tensor_spectral}, there are at most $m^2$ terms involving $M\nrm{\ortc^*}^2$, and at most $2^m$ terms involving $M\nrm{\ortc^*}^i$ with $i\in\{3,\dots,m-1\}$.
Assuming $\nrm{\xast-\xz}\leq m^2 /2^{m} \leq 1$ and recalling that by \eqref{eq:lb_1}, $\nrm{\ortc^*} \leq \nrm{\xast-\xz}$,
\beq
\sum_{i=2}^{m-1} \tbinom{m-1}{i}\Big\|\T(I,\underbrace{\xv,\ldots,\xv}_{m-i-1},\underbrace{\ortc^*,\ldots,\ortc^*}_{i})\Big\| \leq m^2 M \nrm{\ortc^*}^2 +  2^m M\nrm{\ortc^*}^3 \leq   2 m^2 M \nrm{\ortc^*}^2.
\eeq
For the first term in \eqref{eq:Delta_norm_oncm}, recalling the definition of $\beta^*$ in \eqref{eq:beta_ast},
\beqn
\nonumber
|\beta^*| &=& |\mu(\xv)-\alpha^{m-2}\lambda^\ast| = |\T(\alpha \xast-\ortc^*,\ldots,\alpha \xast-\ortc^*)-\alpha^{m-2}\lambda^\ast| 
\\
\label{eq:lambda_diff}
&\leq& | \alpha^{m-2}\lambda^\ast(\alpha^2-1)|
+ \sum_{j=1}^m \alpha^{m-j} \tbinom{m}{j}  \Big|\T(\underbrace{\xast,\ldots,\xast}_{m-j},\underbrace{\ortc^*,\dots,\ortc^*}_{j \text{ times}})\Big|.
\eeqn
For the first term in \eqref{eq:lambda_diff}, note that $\alpha^2-1 = \nrm{\ortc^*}^2$, $|\alpha| \leq 1$ and $|\lamast| \leq M$,
hence 
\beq
| \alpha^{m-2}\lambda^\ast(\alpha^2-1)| \leq M \nrm{\ortc^*}^2.
\eeq
Since $\xast$ is an eigenvector and $(\ortc^*)^T \xast = 0$, all terms in the sum in \eqref{eq:lambda_diff} with $j=1$ vanish,
\beq
\T(\xast,\ldots,\xast,\ortc^*) = \lambda^\ast (\ortc^*)^T \xast=0. 
\eeq
Bounding each term in the sum in \eqref{eq:lambda_diff} with $j\geq 2$ by \eqref{eq:oncm_tensor_spectral}, there are at most $m^2$ terms involving $M\nrm{\ortc^*}^2$, and at most $2^m$ terms involving $M\nrm{\ortc^*}^j$ with $j\in\{3,\dots,m\}$.
Since $\nrm{\ortc^*}\leq m^2/2^m$, the first term in \eqref{eq:Delta_norm_oncm} is thus bounded by
\beq
|\beta^*|\cdot \nrm{\ortc^*} \leq M \nrm{\ortc^*}^3 +  2^m M \nrm{\ortc^*}^3 \leq 2 M m^2 \nrm{\ortc^*}^2.
\eeq
It follows that
\beq
\nrm{\tD(\xz, \ortc^*, \beta^*)} \leq 4 M m^2 \nrm{\ortc^*}^2.
\eeq
Inserting this bound into \eqref{eq:zast_zz},
\beq
\nrm{\bz^* - \zz} \leq \frac{8 M m^2}{\g} \nrm{\ortc^*}^2.
\eeq
Inserting this into \eqref{eq:bx_bxast},
\beqn
\label{eq:ub_2}
\|\ortc^*-\ortcz\| \leq \|\bz^* - \zz\| + \|\ortc^*\|^2 \leq \Big(\frac{8 M m^2}{\g} +1\Big) \nrm{\ortc^*}^2.
\eeqn
Note that if $\nrm{\xast-\xz}\le \eps_2=\min\{1,(4(\frac{8 M m^2}{\g} +1))^{-1/2}\}$, then $\alpha \geq 1/2$.
By \eqref{eq:lb_1}, $\|\ortc^*\|\le \eps_2$ as well.
Thus, \eqref{eq:ub_2} implies $\|\ortc^*-\ortcz\|\leq 1/4$ as required by Lemma \ref{lem:ONCM_aux}.
Under this condition, \eqref{eq:oncm_xast_xo} implies
\beq
\nrm{\xast-\xo} \leq 8\nrm{\ortc^* - \ortcz} \leq 
8\Big(\frac{8 M m^2}{\g}+1\Big) \nrm{\ortc^*}^2
.
\eeq
%
Combining the last two bounds we obtain
\beqn
\label{eq:err_th_ort}
\err_0 = \frac{\nrm{\xast-\xo}}{\nrm{\xast-\xz}^2}
\leq 
8\Big(\frac{8 M m^2}{\g}+1\Big).
\eeqn
As an interim summary, if $\nrm{\xast-\xz}\le \min\{\eps_1,\eps_2,m^2/2^m\}=\eps_0$, then \eqref{eq:err_th_ort} holds. 
The rest of the proof follows by induction as in the proof of Theorem \ref{lem:simple_converge}.
\end{proof}





\section{\normalfont Proof  of Proposition \ref{lem:number_real_eigenpairs}}
\label{sec:example_proof}
First, we prove an auxiliary lemma concerning the structure of the eigenvectors of $\T_\weight$. 
Recall $l=\floor{n/2}$. For any subset $\A \subseteq \{1,\ldots, l\}$ define the following two $n$-dimensional vectors,
\[
\Av= \sum_{i \in \A} \bm e_i \quad\text{and}\quad \Avc= \sum_{i \notin \A}\bm e_i.
 \quad 
\]

\begin{lemma}\label{lem:Tbeta_ev}
There is a function 	
	$\alpha(\weight,|\A|):\R\times\mathbb N\to\R$ such that all eigenvectors of $\T_\weight$ are of the form
\[
\xast(\weight,\A) \propto  \alpha(\weight,|\A|)\Av  + \Avc
,\quad 
\]
	\end{lemma}
\begin{proof}
Let $\xast= \sum_{i=1}^n \alpha_i \bm e_i$ be an eigenvector of $\T_\weight$ with eigenvalue $\lamast$. 
To prove the lemma it suffices to show that the coefficients $\alpha_1,\ldots,\alpha_n$ can attain at most two distinct values.
Applying mode product to $\T_\weight$ with $\xast$, 
\begin{equation}\label{eq:T_w_mult}
\T_\weight(I,\xast,\xast) = \sum_{i=1}^n \alpha_i^2 \bm e_i + \weight\Big(\sum_{i=1}^n \alpha_i\Big)^2 \bm 1 
= \sum_{i=1}^n (\alpha_i^2+ \weight\bar \alpha^2)\bm e_i,
\end{equation}
where $\bar \alpha = \sum_{i=1}^n \alpha_i$.
Since $(\xast,\lamast)$ is an eigenpair it satisfies, 
\begin{equation}\label{eq:T_w_eigenpair}
	\sum_{i=1}^n (\alpha_i^2+ \weight\bar \alpha^2)\bm e_i = \lamast \sum_{i=1}^n \alpha_i \bm e_i.
\end{equation}
Multiplying both sides of \eqref{eq:T_w_eigenpair} with $\bm e_i^T$ gives,
\begin{equation}\label{eq:alpha_i_quadratic}
(\alpha_i^2+ \weight\bar \alpha^2) = \lamast \alpha_i, \qquad \forall i\in \{1,\ldots,n\}.
\end{equation}
Subtracting Equations (\ref{eq:alpha_i_quadratic}) with $j\neq i$,
\[\alpha_i^2-\alpha_j^2=(\alpha_i+\alpha_j)(\alpha_i-\alpha_j) = \lamast (\alpha_i-\alpha_j).\]
We thus conclude that for any $j \neq i$ either $\alpha_j=\alpha_i$ or $\alpha_j=\lamast-\alpha_i$.
It follows that the set $\{\alpha_1,\ldots,\alpha_n\}$ contains up to $2$ distinct values satisfying \eqref{eq:alpha_i_quadratic}. 
\end{proof}

The first part of Proposition \ref{lem:number_real_eigenpairs} determines the number of real eigenvectors for $\T_\weight$. 
Following lemma \ref{lem:Tbeta_ev}, let $\xast = \Avc + \alpha \Av$ be proportional to some eigenvector of $\T_\weight$. 
By Eq. \eqref{eq:T_w_mult},
\[
\T_\weight(I,\xast,\xast) = (\alpha^2+\weight\bar \alpha^2) \Av+(1+\weight\bar \alpha^2) \Avc. 
\]
Since $\xast=\alpha \Av +\Avc$ is proportional to an eigenvector of $\T_\weight$, 
\[
\alpha = \frac{\alpha^2+\weight\bar \alpha^2}{1+\weight\bar \alpha^2}.
\]
or equivalently,
\begin{equation}\label{eq:alpha_ratio}
\alpha(1-\alpha)=\weight(1-\alpha)\bar \alpha^2.
\end{equation}
One solution to Eq. \eqref{eq:alpha_ratio} is $\alpha=1$, which corresponds to the eigenvector $\xv=\frac{1}{\sqrt{n}}\bm 1$. For $\alpha \neq 1$ we replace $\bar \alpha$ with,
\[
\bar \alpha = \sum_{i=1}^n \alpha_i = \alpha |\A|+ (n-|\A|).
\]
The result is the following quadratic equation, 
\begin{equation}\label{eq:quad}
\weight|\A|^2\alpha^2 + (2\weight|\A|(n-|\A|)-1)\alpha+\weight(n-|\A|)^2 =0.
\end{equation}
The solutions to Eq.\ \eqref{eq:quad} determine, up to a normalizing factor, the eigenvectors (both real and complex) of $\T_\weight$.
Due to the problem's symmetry 
we may charaterize \textit{all} real eigenvectors by computing the  solutions to \eqref{eq:quad} only for subsets $\A$ with $0 \leq |\A|\leq l$.  
Consider the discriminant $\mathcal D(\weight,|\A|)$ of the quadratic equation \eqref{eq:quad},
\[
\mathcal D(\weight,|\A|) = (2\weight|\A|(n-|\A|)-1)^2-4\weight^2|\A|^2(n-|\A|)^2 = 1 - 4\weight|\A|(n-|\A|).
\]
For a given $\A$, the number of real solutions to \eqref{eq:quad} is
\[\begin{cases}
2 & \weight< \frac{1}{4|\A|(n-|\A|)} \\
1 & \weight= \frac{1}{4|\A|(n-|\A|)} \\
0 & \weight> \frac{1}{4|\A|(n-|\A|)}.
\end{cases}\]
Hence, the number of real eigenpairs decreases at specific thresholds.
The smallest threshold corresponds to $|\A|=l$ and is given by $\weight_0 = \frac{1}{4l(n-l)}$. When $\weight <\weight_0$,  there are $2$ real solutions to Eq. \eqref{eq:quad} for all subsets $1\leq|\A|\leq l$. 
So the total number of solutions is equal to $2$ times the number of distinct subsets,  
\[
N(\weight<\weight_0) =1+2\sum_{i=1}^l \binom{n}{i} = 2^n-1,
\]
where we add one to account for $\frac{1}{\sqrt{n}} \bm 1$, corresponding to $\A=\emptyset$. Note that this is also the bound on the number of eigenvectors of a generic cubic tensor, see \cite{cartwright2013number}.
When $\weight=\weight_0$, $\mathcal D(\weight_0,|\A|=l)=0$. 
In this case $N(\weight)$ is composed of two eigenvectors for all subsets $1 \leq |\A| \leq l-1$ and one eigenvector for  each subset of size $|\A|=l$, 
\begin{equation}
N(\weight=\weight_0) = 1+2\sum_{j=1}^{l-1} \binom{n}{j} + \binom{n}{l}.
\end{equation}  
For $\weight_0<\weight<\weight_1$, there are no real solutions of Eq. \eqref{eq:quad} for subsets of size $|\A|=l$. The number of real solutions is therefore,
\[
N(\weight_0<\weight<\weight_1) =1+2\sum_{j=1}^{l-1}\binom{n}{j}.
\]
Repeating the argument for increasing values of $\weight$ we obtain $N(\weight)$ as given in the proposition's statement.

We now prove the second part of the proposition, stating that at the thresholds $\weight_i=\frac{1}{4(l - i)(n-l+i)}$, $\binom{n}{l-i}$ of the eigenvectors are not Newton-stable.
In this case $\mathcal D(\weight_i,l-i)=0$ and only one (real) solution to (\ref{eq:quad}) exists for each $\A$ with $|\A|=l-i$.
Solving \eqref{eq:quad} for $\weight=\weight_i$, we find that the $\binom{n}{l-i}$
eigenpairs $(\xast(\A),\lamast(\A))$ with $|\A| = l-i$ are
\[\xast(\A) = 
\sqrt{\frac{1}{n|\A|(n-|\A|)}} \big((n-|\A|)\Av+|\A|\Avc\big), \qquad \lamast(\A)= \sqrt{\frac{n}{|\A|(n-|\A|)}}.\]
We show that each such eigenpair is not Newton-{\em stable}. To do so, we prove that the projected Hessian $H_p(\xast)$ is rank deficient. First we compute the Hessian $H(\xast)$. 
Abbreviate $b = \sqrt{\frac{1}{n|\A|(n-|\A|)}}$ and note that $\lamast = nb$. Then,
\begin{equation}
H(\xast) = 2\T(I,I,\xast)-\lamast I = b \bigg( (n-2|\A|)\diag(\Av) +(2|\A|-n) \diag(\Avc)+ \bm 1 \bm 1^T\bigg).
\end{equation}
Consider the vector $\bm v = \Av-\Avc$. 
A simple calculation yields $\bm v^T H(\xast) \bm v = 0$.
Since $\bm v$ is orthogonal to $\xast$,
$H_p(\xast)$ is rank deficient and $(\xast,\lamast)$ is not Newton-stable.



\section{\normalfont Convergence to eigenvectors which are not Newton-stable}
\label{app:unstable}
In this section we present a detailed empirical study of the convergence properties of O--NCM.
As discussed in Section \ref{sec:newton_based}, the main property that governs the convergence of O--NCM to an eigenpair $(\xast,\lambda^*)$ is the spectral structure of the projected Hessian at the eigenvector, $H_p(\xast)$.
As shown in Theorem 2, when $H_p(\xast)$ is full rank, O--NCM converges in a quadratic rate to $\xast$ given a sufficiently close initial point. 
When $\xast$ is isolated but $1 \leq \text{rank}(H_p(\xast)) < n$, the convergence rate may be less than quadratic. When $H_p(\xast)=0$ and/or $\xast$ is non-isolated, full convergence to $\xast$ is 
not always observed. 
These properties are summarized in table \ref{table:convergence}. 

\begin{table}[h!]
	\centering
	\begin{tabular}{ | c | c | c | c| }
		\hline
		& $H_p(\xast)$ full rank & $H_p(\xast)$ rank deficient  & $H_p(\xast) = 0$\\ 
		\hline
		Isolated & Quadratic convergence & Slow convergence &No guarantees\\ 
		\hline
		Non-isolated  & --- & No guarantees &No guarantees\\
		\hline	
	\end{tabular}
	\caption{O--NCM convergence properties to an eigenvector $\xast$.}
	\label{table:convergence}
\end{table}


We illustrate these convergence properties via two examples. 

\begin{itemize}
	\item[(a)] Consider the tensor $\T$ with order $m=3$ and dimensionality $n=6$ of Example $5.8$ in \cite{cartwright2013number}, corresponding to the homogeneous polynomial
\[
\mu(\xv) = x_1^4x_2^2+x_1^2x_2^4+x_3^6 -x_1^2x_2^2x_3^2.
\]
This tensor has a total of $17$ real eigenpairs. Six of them correspond to a $\lambda=0$ eigenvalue, two of which are not Newton-stable with $\text{rank}(H_p(\xast))=2$. The rest are Newton-stable. 
Figure \ref{fig:unstable_converge} shows the value of $\|\xk-\xast\|$ as a function of the iteration $k$ for one eigenvector that is Newton stable and one that is not. While the convergence to the stable eigenvector is quadratic, the convergence to the point which is not Newton-stable point is much slower. 
\begin{figure}[t]\label{fig:unstable_converge}
\centering
\includegraphics[width=0.5\linewidth]{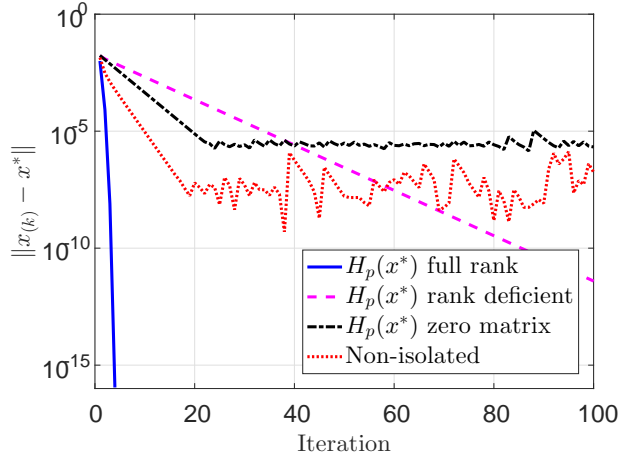}	
\caption{Convergence properties of O-NCM to eigenvectors with different stability properties.}
\end{figure}
\item[(b)] Consider the tensor $\T \in \R^{6 \times 6 \times 6 \times 6}$ of example $6.4$ in \cite{li2013z}, corresponding to the homogeneous polynomial 
\[
\mu(\xv) = \sum_{i=1}^6 \sum_{j>i} (x_j-x_i)^4.
\]
There are a total of $42$ isolated eigenvectors, including one that corresponds to an eigenvalue $\lambda=0$. The projected Hessian $H_p(\xast)$ for this vector is equal to a zero matrix. 
As can be seen in Figure \ref{fig:unstable_converge}, in this case the O-NCM does not fully converge.

In addition, there are also infinitely many non-isolated eigenvectors corresponding to an eigenvalue $\lambda=4.5$. The projected Hessian of these eigenpairs is a rank deficient (though non zero) matrix. For example, any vector of the form 
\begin{equation}\label{eq:non_isolated_ev}
\xast = [a, \,\, a, \,\, b, \,\, b, \,\, -(a+b), \,\, -(a+b)]^T, \qquad a,b \in \R
\end{equation}
is proportional to a non-isolated eigenvector.
Note that the vectors corresponding to \eqref{eq:non_isolated_ev} form a 2 dimensional subspace. Since these vectors are non-isolated, in this case we measure $\|(I_n-P_{\xast})\xk\|$ instead of $\|\xast-\xk\|$ where $P_{\xast} \in R^{n \times n}$ is the projection matrix onto that subspace. As can be seen in Fig. \ref{fig:unstable_converge} in this case the O-NCM does not converge.
\end{itemize}

\paragraph{Trivial eigenvectors} In some cases, the tensor fibers are spanned by a low dimension subspace. Any vector orthogonal to this subspace is an eigenvector corresponding to an eigenvalue $\lambda=0$, and a projected Hessian $H_p(\xast)$ equal to a zero matrix. This is the case, for instance in example (b) where all fibers are orthogonal to $\xast = [1\ldots,1]^T$. As we have seen, this may cause the O-NCM to slowdown, since the iterations do not converge to these points. 

 A simple pre-processing step is to find these eigenvectors by calculating the subspace of the tensor fibers, namely $\T_{:,i_2,\ldots,i_m},i_2,\ldots,i_m \in [n]$. As a second step, the O-NCM can easily be constrained to that subspace.




\bibliographystyle{plain}
\bibliography{bib_tensor_eigenvectors}

\end{document}